\documentclass[reqno,10pt]{amsart}
\usepackage[left= 1.3 in, right= 1.3 in ,top= 1 in, bottom = 2 in]{geometry}

\usepackage{graphicx}%
\usepackage{multirow}%
\usepackage{amsmath,amssymb,amsfonts}%
\usepackage{amsthm}%
\usepackage{mathrsfs}%
\usepackage[title]{appendix}%
\usepackage{xcolor}%
\usepackage{listings}%
\usepackage[capitalize]{cleveref}

\newtheorem*{theorem*}{Theorem}
\newtheorem*{prop*}{Proposition}
\newtheorem{teo}{Theorem}[section]
\newtheorem{cor}[teo]{Corollary}
\newtheorem{lema}[teo]{Lemma}
\newtheorem{defn}[teo]{Definition}
\newtheorem{prop}[teo]{Proposition}
\newtheorem{ej}[teo]{Example}
\newtheorem{obs}[teo]{Remark}
\newtheorem{thmx}{Theorem}
\newtheorem{innercustomthm}{Theorem}
\newenvironment{customthm}[1]
  {\renewcommand\theinnercustomthm{#1}\innercustomthm}
  {\endinnercustomthm}

\newcommand{\N}{\mathbb{N}}
\newcommand{\Z}{\mathbb{Z}}
\newcommand{\R}{\mathbb{R}}
\newcommand{\C}{\mathbb{C}}
\newcommand{\Q}{\mathbb{Q}}
\newcommand{\T}{\mathbb{T}}

\newcommand{\sF}{\ensuremath{\mathcal{F}}}
\newcommand{\norm}[1]{\left\lVert#1\right\rVert}

\begin{document}

\title[On recurrence for $\Z^d$-Weyl systems]{On recurrence for $\Z^d$-Weyl systems}

\author{Sebasti\'an Donoso}
\address{Departamento de Ingenier\'{\i}a
	Matem\'atica and Centro de Modelamiento Ma\-te\-m\'a\-ti\-co, Universidad de Chile and IRL-CNRS 2807, Beauchef 851, Santiago,
	Chile.} 
\email{sdonoso@dim.uchile.cl}

\author{Felipe Hern\'andez}
\address{Department of Mathematics, Ecole polytechnique fédérale de Lausanne, Rte Cantonale, 1015 Lausanne, Switzerland} 
\email{felipe.hernandezcastro@epfl.ch}

\author{Alejandro Maass}
\address{Departamento de Ingenier\'{\i}a Matem\'atica, Centro de Modelamiento Ma\-te\-m\'a\-ti\-co and Millennium Institute Center for Genome Regulation, Universidad de Chile and IRL-CNRS 2807, Beauchef 851, Santiago,
Chile.}
\email{amaass@dim.uchile.cl}

\thanks{The first and third authors were partially funded by Centro de Modelamiento Matemático (CMM) FB210005, BASAL funds for centers of excellence from ANID-Chile. The first author was partially funded by ANID/Fondecyt/1241346. The third author was funded by grant ICN2021-044 from the ANID Millennium Science Initiative. All authors are part of the ECOS-ANID grant C21E04 (ECOS210033)}

\begin{abstract}
We study the topological recurrence phenomenon of actions of locally compact abelian groups on compact metric spaces. In the case of $\Z^d$-actions we develop new techniques to analyze Bohr recurrence sets. These techniques include finding and manipulating correlations between the coordinates of the set of recurrence. Using this, we show that Bohr recurrence sets are recurrence sets for $\Z^d$-Weyl systems. This family encompasses, for example, all $\Z^d$-affine nilsystems. To our knowledge, this is the first result towards a positive answer to Katznelson question in the case of $\Z^d$-actions.
\end{abstract}



\maketitle

\section{Introduction}\label{Introduction}
The study of the phenomenon of recurrence in ergodic theory has a rich history. Its origins can be traced back to H. Poincaré's work in the late 19th century on the three-body problem, where he discovered that trajectories in certain dynamical systems tend to return to their starting points even in chaotic or unpredictable scenarios. Poincaré's Recurrence Theorem formalizes this observation, stating that for a measure-preserving transformation $T$ on a probability space $(X, \mathcal{X}, \mu)$ and a set $A \in \mathcal{X}$ with $\mu(A)>0$, for almost any point $x \in A$ there exists a sequence $(n_i)_{i\in \N}$ such that $T^{n_i}x \in A$ for all $i \in \mathbb{N}$.

 The recurrence phenomenon has served to prove various deep results in different areas of mathematics. Since the works of Furstenberg \cite{Furstenberg_ergodic_szemeredi:1977}, and Furstenberg and Weiss \cite{Furstenberg_Weiss_top_dynamics_combin_number_theory:1978}, there has been a profound and fruitful connection between dynamical systems and number theory through recurrence properties, and, more specifically, the notion of set of recurrence (see \cref{def:recurrence_set}) playing a central role in both areas.
 Within this context, the class of nilsystems has proved to be particularly important. Nilsystems have been used, for example, in the ergodic context of multiple convergence along arithmetic progressions \cite{Host_Kra_nonconventional_averages_nilmanifolds:2005} and in the topological context to construct explicit examples of sets with multiple recurrence\footnote{A set $R\subseteq \N$ is a a set of topological multiple recurrence if for every $l\in \N$, every minimal topological system $(X,T)$, and every open set $U\subseteq X$ there is $n\in R$ with $U\cap T^{-n}U\cap\cdots \cap T^{-ln}U\neq \emptyset$.} \cite{Frantzikinakis_Lesigne_Wierdl_sets_k-recurrence:2006, Huang_Shao_Ye_nilbohr_automorphy:2016}.

In our study, we focus on sets of Bohr recurrence, which are sets that exhibit recurrence for all rotations\footnote{Rotations are a subclass of nilsystems, see \cref{nil}}. One is naturally led to ask if this is enough to be a set of recurrence. This is typically known as Katznelson's question, popularized by Katznelson \cite{Katznelson_chromatic_numbers_recurrence:2001} for $\mathbb{Z}$-actions.

Note that we avoided framing Katznelson's question as a conjecture, as neither of the potential answers seems more likely than the other. For a negative answer, there are some potential counterexamples that are known to be good for recurrence in circle rotations, but it remains unclear whether they are good for recurrence across all systems \cite{Grivaux_Roginskaya_examples_recurrence_circle:2013, Frantzikinakis_McCutcheon_ergodic_recurrence:2012, Frantzikinakis_Host_Kra_Bohr_recurrence:2024}. For a positive answer, possible strategies include the use of general structure theorems in topological dynamics and separate the analysis for factors with specific properties (such as equicontinuity, weakly mixing, etc).  For instance, every minimal system can be obtained through proximal and equicontinuous extensions from a strictly PI system (for background see \cite{Auslander_minimal_flows_and_extensions:1988} Chapter 14 or \cite{deVries_elements_topological_dynamics:1993} Chapter 6), or by passing through maximal $d$-step nilsystems and their maximal distal factor (see \cite{Gutman_Zhengxing_strict_distal_models_host-kra:2023} and \cite{Host_Kra_nilpotent_structures_ergodic_theory:2018} for background). These decomposition methods have shown promising results for $\Z$, such as lifting Bohr recurrence to arbitrary minimal $\mathbb{Z}$-pronilsystems \cite{Host_Kra_Maass_variations_top_recurrence:2016}.  Unfortunately, Glasscock, Koutsogiannis and Richter constructed in \cite[Theorem B]{Glasscock_Koutsogiannis_Richter_OnKatznelson_skew:2022} an example of a system and a set that is a set of recurrence for its maximal equicontinous factor, but not a set of recurrence for the system itself. So, one cannot always reduce the study of recurrence to a special factor\footnote{We warn the reader that the aformentioned set is not of Bohr recurrence.  It is of recurrence only for the particular factor considered}.

The main purpose of this article is to study Katznelson's question outside the context of $\Z$-actions. The context of locally compact abelian groups has been previously explored in, for instance, \cite{Le_Le_Bohr_sumsets_I:2021,Griesmer_special_cases_Katznelson:2023, Griesmer_Le_Le_Bohr_sumsets_II:2023}. For general locally compact abelian group actions, we begin in \cref{Sec3} defining Bohr recurrence for abelian locally compact groups, using theory of duality and proving that this definition is equivalent to the usual definition using equicontinuous systems. Then, we state Katznelson's question and we finish \cref{Sec3} pointing out that inverse limits and proximal extensions lift recurrence, a result that was proved for $H=\Z$ by Host, Kra, and Maass in
\cite{Host_Kra_Maass_variations_top_recurrence:2016} but that generalizes with minor adjustments for the action of a locally compact abelian group.

In \cref{Sec4} we study sets of Bohr recurrence for $\Z^d$ group actions, introducing the notion of essential sets of Bohr recurrence, which are sets of Bohr recurrence in which every element has no zero coordinates (see \cref{Essential-set-of-Bohr-recurrence}). We give examples and prove many properties of such sets. Then, we introduce the notion of Bohr correlations (see \cref{Bohr-Correlations}), with which we develop techniques to classify sets of $\Z^d$-Bohr recurrence. In particular, a subclass of sets of great importance is the class of sets of Bohr recurrence with the property of complete independence that corresponds to the case in which the coordinates of a set of Bohr recurrence are correlated either by $0$ or by an irrational number (see \cref{Complete-independence}).

Our main result comes in \cref{Sec5}, where we generalize the theorem of Host, Kra, and Maass on minimal $\Z$-nilsystems to $\Z^d$-Weyl systems, which are nilsystems $(X=G/\Gamma,T_1,\ldots,T_d)$ whose connected component of the identity in $G$ is abelian. 

\begin{thmx}\label{Theorem-A}
Let $R\subseteq \Z^d$ be a set of $\Z^d$-Bohr recurrence. Then, for every integer $s\geq 1$, $R$ is a set of recurrence for every minimal $s$-step $\Z^d$-Weyl system.
\end{thmx} 

The key ingredient to prove \cref{Theorem-A} is the following result of independent interest in the study of sets of Bohr recurrence. The proof, together with the necessary definitions, is given in \cref{Sec4}.

\begin{thmx}\label{Theorem-B}
Let $R$ be an essential and ordered set of $\Z^d$-Bohr recurrence with the property of complete independence. Then, for all $r\in \N$ and $\epsilon>0$ there exists a set of Bohr recurrence $R_\epsilon\subseteq R$ such that for any $\mathbf{n}=(n_1,\ldots,n_d)\in R_\epsilon$ and $\mathbf{w}_1,\ldots,\mathbf{w}_d\in \T^r$ there is $y\in \R^r$ such that
\begin{align*}
||y||_{\R^r}<\epsilon ~~ \text{ and }~~ ||(n_1y,\ldots,n_d y)-(\mathbf{w}_1,\ldots,\mathbf{w}_d)||_{\T^{d\cdot r}}<\epsilon.
\end{align*}

\end{thmx}
This result is key to approximate fibers in $\Z^d$-affine $s$-step nilsystems, allowing us to lift recurrence from lower-step nilfactors.

While our method works well within the family of $\Z^d$-Weyl systems, extending it to general nilsystems remains a challenge. Currently,  we are exploring a direct proof using spectral methods for the general case. To be more precise, a generalization of Theorem 1.4 from \cite{Ackelsberg_Richter_Shalom_maximal_spectral_nil:2023} to $\Z^d$-actions would be key for a measurable proof of the general case.

\section{Notation, terminology, and prerequisites}\label{Ch1}

\subsection{Topological Dynamical Systems}
A topological dynamical system, or just an $H$-system, is a pair $(X,H)$ where $X$ is a compact metric space and $H$ is a locally compact group of homeomorphisms of the space $X$ into itself, which for our purposes we will suppose abelian. The action of an element $g \in H$ on a point $x \in X$ is denoted by $g(x)$ or $gx$, and the orbit of a point $x$ under the action of $H$ is represented by $\mathcal{O}_H(x) = \{hx: h \in H\}$ (or simply $\mathcal{O}(x)$ if the acting group is clear).

A topological dynamical system $(X,H)$ is {\em transitive} if there is a point $x\in X$ which orbit $\mathcal{O}_H(x)$ is dense in $X$. We call such points {\em transitive points}. A stronger condition than transitivity is {\em minimality}. A system $(X, H)$ is said to be {\em minimal} if every point in $X$ is transitive. In other words, there are no nonempty closed $H$-invariant proper subsets of $X$. A subset $Y \subseteq X$ is  $H$-invariant if $HY = Y$.

We say that $(Y, H)$ is a {\em subsystem} of $(X, H)$ if $Y$ is a nonempty closed $H$-invariant subset of $X$. An essential fact is that every dynamical system has a minimal subsystem.

A system $(X,H)$ is  \textit{equicontinuous} if for all $\epsilon>0$ there exists $\delta>0$ such that for every $x,y\in X$, $d_X(x,y)<\delta$ implies $d_X(h x, h y)<\epsilon$, for all $h\in H$ ($d_X$ denotes a metric in $X$). A pair $(x,y)\in X\times X$ is called \textit{proximal} if there exists a sequence $(h_n)_{n\in \N}\subseteq H$ such that
$ d(h_nx,h_ny)\to 0$ as $n$ goes to infinity. If there is no such sequence, the pair $(x, y)$ is said to be \textit{distal}. A system is distal if it does not have any nontrivial proximal pairs.

A factor map between topological dynamical systems $(X, H)$ and $(Y, H)$ is an onto continuous map $\pi\colon X \to Y$ such that  $\pi \circ h = h \circ \pi$ for all $h \in H$. In this case, $X$ is called an extension of $Y$ and $Y$ a factor of $X$. If $\pi$ is also injective, we say that $(X, H)$ and $(Y, H)$ are topologically conjugate (or just conjugate), and $\pi$ is called a conjugacy. An extension map $\pi:X\to Y$ is called \textit{proximal} if every pair of point $(x_1,x_2)\in X\times X$ with $\pi(x_1)=\pi(x_2)$ are proximal.

Let $\{(X_n,H)\}_{n\in \N}$ be a sequence of topological dynamical systems such that for each $n\in \N$  there is a factor map $\pi_{n+1}:(X_{n+1},H)\to (X_{n},H)$. The \textit{inverse limit} of $\{(X_n,H)\}_{n\in \N}$ of these factors is defined as the set
$$\Bigl\{ (x_n)_{n\in \N} \in \prod_{n\in \N} X_n ~ :~ \pi_{n+1}(x_{n+1})=x_n~ \text{ for all } n\in \N\Bigr\},$$
equipped with the product topology and the diagonal action of $H$ given by \hfill\break $h(x_n)_{n\in \N}=(hx_n)_{n\in \N}$ for every $h\in H$.

A measure preserving system (m.p.s. for short) $(X, \mathcal{X},\mu, H)$ consists of a probability space $(X, \mathcal{X}, \mu)$ endowed with a $\mathcal{X}$-measurable and measure preserving action of $H$. That is, for any $A \in \mathcal{X}$ and $h\in H$, $hA\in \mathcal{X}$ and $\mu(hA) = \mu(A)$. A set $A \in \mathcal{X}$ is \textit{invariant} if $\mu(A \Delta hA) = 0$ for all $h\in H$. A m.p.s. $(X, \mu, H)$ is called \textit{ergodic} if there are no nontrivial invariant sets, meaning that if $A \in \mathcal{X}$ and $\mu(A \Delta hA) = 0$ for all $h \in H$, then $\mu(A) \in \{0, 1\}$.

\subsection{Nilsystems}\label{nil}
Let $G$ be a group. For $a,b\in G$ we define their commutator as $[a,b]=aba^{-1}b^{-1}$. Given subgroups $A,B\subseteq G$, we write $[A,B]$ to denote the subgroup of $G$ generated by $\{[a,b] : a\in A,b\in B\}$. We define recursively the commutator subgroups $G_j$ of $G$ as:
$$G_1=G,~~~ G_{j+1}=[G,G_j], ~~ j\geq 1. $$
We say that a group $G$ is nilpotent of order $s$ or $s$-step nilpotent if $G_{s+1}=\{e_G\}$, where $e_G$ is the identity of $G$.

Let $G$ be an $s$-step nilpotent Lie group and $\Gamma$ a discrete cocompact subgroup of $G$ (i.e., $\Gamma$ is countable and $G/\Gamma$ is compact). Then, we say that the compact nilmanifold $X=G/\Gamma$ is an \textit{$s$-step nilmanifold.} Let $H\subseteq G$ be a subgroup and consider the action of $H$ given by left translations. Then $(X,H)$ is a $H$-system, also called an $s$-step nilsystem, or to be more precise, an $s$-step $H$-nilsystem. We will always consider the Haar measure $\mu$ of $(X,H)$ (the unique invariant measure for the action of left translations).

\begin{ej}[Rotations]
Let $X$ be a compact abelian group, $H$ be a group, and $\varphi\colon H\to X$ be a homomorphism. Consider the action of $H$ on $X$ induced by $\varphi$, that is, $h\cdot x= \varphi(h)+x$, for all $h\in H$ and $x\in X$, where the group operation of $X$ is written in additive notation. Then $(X,H)$ is a $1$-step nilsystem, and this system is usually called a rotation.  
\end{ej}

For a nilpotent Lie group $G$ we denote by $G_0$ the connected component of the identity $e_G$ in $G$. Then, $G_0$ is an open normal subgroup of $G$ (see \cite[Section 4.1]{Host_Kra_Maass_variations_top_recurrence:2016}). Various representations of nilsystems are available, some of them specifically tailored to address specific problems and contexts.  One possible choice allows us to assume that $G$ is simply connected, which means
that $G_0$ is simply connected (see \cite[Chapter 10, Theorem 13]{Host_Kra_nilpotent_structures_ergodic_theory:2018}). Another property of nilsystems is that they are all distal (\cite[Chapter 4, Theorem 3]{Auslander_Green_Hahn_flows_homogeneous:1963}).

Now we give some important structural properties of nilsystems that will be useful in the next chapters.

\begin{teo}[\cite{Parry_ergodic_affine_nil:1969}]\label{Parry-theorems}
Let $(X=G/\Gamma, H)$ be an $s$-step nilsystem. 
\begin{enumerate}
\item The groups $G_j$, $j\geq 2$, are connected. In particular, $G_j\subseteq G_0$.
\item     Let $N=\langle G_0, H\rangle$ and $Z= X/[N,N]$. Then, the action of $H$ is ergodic with respect to its unique invariant measure if and only if $X$ is minimal with respect to the action of $H$, and if and only if $Z$ is minimal with respect to the action of $H$.
\end{enumerate}
\end{teo}
\begin{prop}[\cite{Host_Kra_nilpotent_structures_ergodic_theory:2018}]\label{open-nilmap}
    Let $(X=G/\Gamma,H)$ be an $s$-step nilsystem. Then the canonical factor $\pi\colon X\to X/G_s$ is open.
\end{prop}
\begin{cor}[\cite{Leibman_pointwise_conv_polynomial_nil:2005}]\label{connected-implies-totally}
Let $(X=G/\Gamma, H)$ be an $s$-step nilsystem with $X$ connected. Then, the action of $H$ is ergodic with respect to its unique invariant measure on $X$ if and only if it is ergodic with respect to its unique invariant measure on the maximal torus factor $X/[G_0,G_0]$ of $X$. 
\end{cor}

The next theorem is a well-known result in the field of nilsystems and it is a direct consequence of \cref{Parry-theorems} \textit{(2)}.

\begin{teo}\label{teo11HKM}
A nilsystem $(X,H)$ is ergodic with respect to its unique invariant measure  if and only if it is minimal.
\end{teo}

\subsection{Weyl Systems}\label{Weyl-section}
In this section, we introduce a central object of this work, the Weyl systems. Weyl systems are usually defined for $\Z$-actions (see \cite{Bergelson_Leibman_Lesigne_complexities_poly_Weyl_systems:2007,Frantzikinakis08,Kuca_notions_complexity_poly:2023} for definitions and motivations), but we extend its definition to $\Z^d$-actions in the natural way. Recall that an action of the additive group $\Z^d$ is given by $d$ commuting homeomorphisms of the space. 

\begin{defn}
 Let $(X=G/\Gamma, T_1,\ldots,T_d)$ be an $s$-step $\Z^d$-nilsystem. We say that $(X=G/\Gamma, T_1,\ldots,T_d)$ is a {\em $\Z^d$-Weyl system} if $G_0$ is abelian.  
\end{defn}

 It is a known fact that for any $s$-step nilsystem $(X, H)$, the closed orbit of any point $x \in X$ remains an $s$-step nilsystem (see for instance \cite[Theorem 1.3]{Shah_inv_measures_orbits_homogeneous_spaces_unipotent:1998}  or \cite[Section 3.2, Chapter 10]{Host_Kra_nilpotent_structures_ergodic_theory:2018}). This property is also true for Weyl systems. To be precise, if $(X=G/\Gamma, T_1,\ldots,T_d)$ is a $\Z^d$-Weyl system and $(Y=G^Y/\Gamma^Y, T_1,\ldots,T_d)$ is an orbit of $X$, then in the proof of Theorem 1.3 in \cite[Chapter 11, Theorem 9]{Host_Kra_nilpotent_structures_ergodic_theory:2018} it is shown that $G_0^Y$ is a subgroup of $G_0$ and therefore $G_0^Y$ is also abelian.
 
 A concrete example of Weyl systems are affine systems. An $s$-step $\Z^d$-nilsystem $(X,T_1,\ldots,T_d)$ is said to be {\em $\Z^d$-affine} if $X=\T^r$ for some $r\geq 1$ and the transformations $T_1,\ldots,T_d$ are defined by $T_i(x)=\boldsymbol{A}_ix+\boldsymbol{\alpha}_i$, where $\boldsymbol{A}_1,\ldots,\boldsymbol{A}_d$ are commuting unipotent matrices and $\boldsymbol{\alpha}_1,\ldots,\boldsymbol{\alpha}_d\in \T^r$. We will consider the representation $X=G/\Gamma$ where $(G,\circ)$ consists of the group of transformations of $\T^r$ generated by the group $\Gamma$ generated matrices by $\{\boldsymbol{A}_{i}\}_{i=1}^d$, and the group of translations of $\T^r$. Thus, every element of $g\in G$ is given by a transformation $x\in X \to \boldsymbol{A}_gx+\boldsymbol{\alpha}_g$, where $\boldsymbol{A}_g\in \Gamma$ and $\boldsymbol{\alpha}_g\in \T^r$. For two elements $g,h\in G$, the commutator $[g,h]$ is given by the transformation $x\in X \to x+(\boldsymbol{A}_g-\boldsymbol{I})\boldsymbol{\alpha}_h-(\boldsymbol{A}_h-\boldsymbol{I})\boldsymbol{\alpha}_g$. In consequence, we can assume without loss of generality that $G_j\subseteq \T^r$ for each $j\geq 2$.
 
 The following result by Frantzikinakis and Kra allows one to show that the family of connected $\Z^d$-Weyl systems corresponds to the family of affine systems.
\begin{prop}[{\cite[Proposition 3.1]{Frantzikinakis_Kra_averages_product_integrals:2005}}]\label{affine-equals-G0-abelian}
 Let $X = G/\Gamma$ be a connected nilmanifold such that $G_0$ is abelian. Then, any nilrotation $T_a(x) = ax$ on $X$ is  topologically conjugate to a unipotent affine transformation on some finite dimensional torus \footnote{In \cite{Frantzikinakis_Kra_averages_product_integrals:2005}, the result is stated for measurable isomorphism, but the same proof gives topological conjugacy.}.   
\end{prop}
By careful examination of the argument in \cite{Frantzikinakis_Kra_averages_product_integrals:2005} one notices that the homeomorphism $X\to \T^d$ constructed in the proof of \cref{affine-equals-G0-abelian} does not depend on the transformation at all. We sketch the argument here for completeness. The proof of \cref{affine-equals-G0-abelian} relies on finding a homeomorphism $X\to G_0$. For this, as $X$ is connected, it can be identified with $G_0\Gamma/\Gamma$, and we may further assume that $G_0\cap \Gamma=\{e_G\}$ since it is possible to replace $G$ by $G/\Gamma\cap G_0$ and $\Gamma$ by $\Gamma/\Gamma\cap G_0$ (because of $\Gamma\cap G_0$ being a normal group in $G$). This reduction makes $G_0$ a connected compact abelian group (and thus, isomorphic to a finite dimensional torus). In addition, for each $g\in G$ there is a unique representation $g=g_0\gamma$ with $g_0\in G_0$ and $\gamma \in \Gamma $, and thus the projection $\phi:X\to G_0$, $\phi(g\Gamma)=g_0$ is a well defined homeomorphism. Noticing that $\phi$ transforms rotations into affine unipotent transformations, we can apply the previous theorem to each transformation of a $\Z^d$-nilsystem.

We will observe in subsequent sections that any minimal $s$-step $\mathbb{Z}^d$-nilsystem can be regarded as a union of connected $s$-step $\mathbb{Z}^d$-nilsystems. Consequently, $\mathbb{Z}^d$-Weyl systems are finite unions of $\mathbb{Z}^d$-affine nilsystems.

\section{Bohr Recurrence for locally compact abelian groups}\label{Sec3}
This section focuses on the concept of recurrence for systems and families of systems, together with the notion of Bohr recurrence, which is established for locally compact abelian groups using Pontryagin duality. We extend many known properties for $G=\Z$. Although these properties might be considered classical, we were unable to find a proper reference in the literature, so we include them for the sake of completeness. Finally, we state Katznelson's question in this context and generalize the fact that proximal extensions and inverse limits lift up recurrence. 

\subsection{Bohr Recurrence}\label{2.1}
For a system $(X,H)$ and sets $U,V\subseteq X$ we denote
$N_H(V,U)=\{h\in H: V\cap h^{-1}U\neq \emptyset\}.$
In the case $U=V$ we just write $N_H(U)=N_H(U,U)$, and in the case $V=\{x\}$ we write
$N_H(x,U)=\{h\in H :  hx\in U\}.$

\begin{defn} \label{def:recurrence_set}
A set $R\subseteq H$ is a \textit{set of recurrence for a system $(X,H)$} if for any nonempty open set $U\subseteq X$, $R\cap N_H(U)\neq \emptyset$. Additionally, if $\sF$ is a family of $H$-systems, a set $R\subseteq H$ is a \textit{set of recurrence for the family $\sF$} if for any minimal system $(X,H)$ in the family $\sF$, $R$ is a set of recurrence for $(X,H).$
\end{defn}
When we say that $R\subseteq H$ is a \textit{$H$-set of recurrence} without specifying any particular family, we mean that $R$ is a set or recurrence for all minimal $H$-systems.

\begin{obs}
    If $R$ is a set of recurrence for a system $(X,H)$, then there exist $x\in X$ and a sequence $(h_n)_{n\in\N}\subseteq R$ such that $h_nx\to x$ as $n$ goes to $\infty$. Indeed, for every $n\in\N$, the set $A_n=\{x\in X:\inf_{h\in R} d(hx,x)<1/n\}$ is open and dense, and therefore any point in the dense $G_{\delta}$ set $\cap_{n\in \N} A_n$ satisfies the aforementioned property. 
\end{obs}

Now, we introduce the generalities for studying sets of recurrence in the family of rotations. We begin by discussing duality in locally compact abelian groups in order to define Bohr recurrence. For the purposes of this discussion, we will denote the unit circle by $\mathbb{S}^1=\{z\in \C  :  |z|=1\}$.

Given a locally compact abelian group $H$, a character of $H$ is a continuous homomorphism $\chi\colon H\to \mathbb{S}^1$, and the set $\widehat{H}$ of characters of $H$ is an abelian group under pointwise multiplication. For a compact subset $F\subseteq H$ and $\epsilon>0$, let
$P(F,\epsilon)=\{\chi \in \widehat{H}  :  |\chi(h)-1|<\epsilon, \forall h\in F\}$. The sets $P(F,\epsilon)$, where $F$ and $\epsilon$ range over all compact sets and positive numbers respectively, is an open basis at the identity $e_{\widehat{H}}$, and $\widehat{H}$ is a locally compact abelian topological group (\cite[Theorem 23.15]{Hewitt_Ross_harmonic_analysis_I:1979}). 

Let $H$ be a locally compact abelian group and consider $\chi_1,\ldots,\chi_k\in \widehat{H}$. The Bohr neighborhood of $e_H$ or just $0$ in $H$ determined by $\{\chi_1,\ldots,\chi_k\}$ and radius $\epsilon>0$  is given by
$$\text{Bohr}(\chi_1,\ldots,\chi_k;\epsilon)=\{h\in H: |\chi_i(h)-1|<\epsilon,~~\forall i\in \{1,\ldots,k\}\}.$$

\begin{obs}\label{Dual-obs}
If $H$ is compact, then as $\widehat{\widehat{H}}$ can be identified with $H$  (see  \cite[Theorem 24.3]{Hewitt_Ross_harmonic_analysis_I:1979}), the topology generated by the basis of $e_H$ given by the sets $Bohr(\chi_1,\ldots,\chi_k;\epsilon)$, for $\chi_1,\ldots,\chi_k\in \widehat{H}$ and $\epsilon>0$, coincides with the topology in $H$.
\end{obs}

Let $H$ be a locally compact abelian group. We say that $V\subseteq H$ is a $H$-Bohr$_0$ set if $V$ contains a Bohr neighborhood of $0$. Additionally, we say that $W\subseteq H$ is a set of $H$-Bohr recurrence (denoted as $W\in H$-Bohr$_0^*$) if for every $V\in \text{Bohr}_0$, $W\cap V\neq \emptyset$. We will usually drop the prefix $H$- from $H$-Bohr$_0$, $H$-Bohr$_0^*$, and $H$-Bohr recurrence if the locally compact abelian group $H$ is clear from the context. The family of $H$-Bohr$_0$ sets is a filter, which means that it is upward closed and closed under intersections. Additionally, the family of $H$-Bohr$_0^*$ sets is \textit{partition regular}, this is, for every $A\in H\text{-Bohr}_0^*$, if we partition $A$ into $N$ sets $(A_n)_{n=1}^N$, then there is an $n\in \{1,\ldots,N\}$ such that $A_n\in H\text{-Bohr}_0^*$. This is well-known for the case $G=\Z$ (see for example \cite{Glasscock_Koutsogiannis_Richter_OnKatznelson_skew:2022}) but the proof is identical for the general case.

A classic example of a $H$-Bohr$_0$ set which is not trivial is the following.

\begin{prop}\label{RotationsAreBohr}
Let $(X,H)$ be a rotation and consider $x\in X$ and $U\subseteq X$ a neighborhood of $x$. Then, $N(x,U)$ is a Bohr$_0$ set.
\end{prop}
\begin{proof}
Without loss of generality we may assume that $x=e_X$, given that every neighborhood of $x$ is a translation of a neighborhood of $e_X$. Let $\varphi\colon H\to X$ be the homomorphism associated to the action of $H$ on $X$. For an open neighborhood $U$ of $e_X$ we have to prove that
$N(e_X,U)=\{h\in H : \varphi(h)\in U\}=\varphi^{-1}(U)$
is a $H$-Bohr$_0$ set. By \cref{Dual-obs},  there exist $\chi_1,\ldots,\chi_k\in \widehat{X}$ and $\epsilon>0$ are such that $Bohr(\chi_1,\ldots,\chi_k;\epsilon)\subseteq U$ from which we deduce $Bohr(\chi_1\circ \varphi,\ldots ,\chi_k\circ \varphi; \epsilon)\subseteq \varphi^{-1}(U)$. This means that $N(e_X,U)$ is a $H$-Bohr$_0$ set and we conclude.
 \end{proof}
As any orbit in an equicontinuous system is, modulo conjugacy, a minimal rotation, we can extend \cref{RotationsAreBohr} to any equicontinuous system. In particular, we have that sets of Bohr recurrence are precisely the family of sets of recurrence for equicontinuous systems and rotations.

\subsection{Katznelson's question}\label{Ch3}
In this section, we state Katznelson's question regarding general group actions.\\

\textit{Katznelson's question:} Given an action $H$, is every set of Bohr recurrence a set of (topological) $H$-recurrence?\\

Notice that in the context of nonabelian groups, one could still formulate Katznelson's problem replacing Bohr recurrence by equicontinuous recurrence. That is, say that $S\subseteq H$ is a set of equicontinuous recurrence if for any minimal equicontinous system $(X,H)$ and nonempty open set $U\subseteq X$, there exists $g\in S$ such that $U\cap g^{-1}U\neq \emptyset$. For abelian groups, as we have seen, equicontinuous recurrence is equivalent to Bohr recurrence. However, one key advantage of equicontinuous recurrence is that it does not rely on the theory of characters, which is not readily available for nonabelian groups. The Katznelson's problem in this context would be asking if sets of equicontinuous recurrence are necessarily sets of recurrence. It is easy to construct groups for which this formulation of Kaztnelson problem is false, as the following example shows for instance.
\begin{ej}
  Consider a nonrecurrent map $T$ on the unit circle $\mathbb{S}^1$ (e.g., $T(x)=x^2$) and let $S$ be an irrational rotation. We have that $(\mathbb{S}^1,\langle T,S\rangle)$ is a minimal system and the set $\N \times \{0\}$ (here we identify this set with $\{T^n:n\in\N\}$, since all powers of $T$ are different)  is a set of equicontinous recurrence but not of recurrence for the system $(\mathbb{S}^1,\langle T,S\rangle)$.
\end{ej}
 For that reason we restrict ourselves to abelian groups and focus our attention on $\Z^d$, $d\geq 2$ later on. As we pointed out in the introduction, a strategy to give a positive answer to Katznelson's question for $\Z$-actions is to prove that Bohr recurrence can be lifted up through certain types of extensions for particular classes of systems. This idea also holds for $H$-actions in our context, and the following propositions show that recurrence can be lifted up through inverse limits and proximal extensions. The proof for inverse limits is elemental and the proof for proximal extensions is a straightforward adaptation of the original argument in \cite[Proposition 3.8]{Host_Kra_Maass_variations_top_recurrence:2016}, so we will omit them.
 
\begin{prop}\label{inverse-limits-preserve-recurrence}
Let $\{(X_n,H)\}_{n\in \N}$ be a collection of minimal topological dynamical systems with factor maps $\pi_{n+1}\colon (X_{n+1},H)\to (X_{n},H)$ for $n\in \N$. Suppose that $R$ is a set of recurrence for $(X_n,H)$ for all $n\in \N$. Then, it is also a set of recurrence
for the inverse limit of these systems with respect to the factors.  
\end{prop}

\begin{prop}\label{Proximal-Theorem}
Let $\pi\colon (X,H)\to (Y,H)$ be a proximal extension between minimal systems and let $R$ be a set of recurrence for $(Y,H)$. Then, $R$ is a set of recurrence for $(X,H)$.
\end{prop}

Unlike proximal extensions and inverse limits, it is not easy to generalize the fact that sets of Bohr recurrence are set of recurrence for $s$-step nilsystems. In the next sections, we focus on doing this for the family of $\Z^d$-Weyl systems by developing certain techniques over sets of $\Z^d$-Bohr recurrence.

The following results show that Bohr recurrence is preserved under homomorphisms of locally compact groups. We present these in generality, but they will be mostly used in the case of $\Z^d$.  Related results can be found in \cite{Le_Le_Bohr_sumsets_I:2021,Griesmer_special_cases_Katznelson:2023}. The following can be found in \cite[Proposition 2.10]{Donoso_Le_Moreira_Sun_additive_mult_corr:2023}, with the Bohr case being a straightforward adaptation of that proof.

\begin{prop} \label{prop:morphism_recurrence} Let $\phi\colon H\to H'$ be a group homomorphism between the locally compact abelian groups $H$ and $H'$, and let $R\subseteq H$ be a set of $H$-recurrence ($H$-Bohr recurrence). Then 
$\phi(R)$  is a set of $H'$-recurrence (resp $H'$-Bohr recurrence). 
\end{prop}

Note that if $H$ and $H'$ are locally compact abelian groups with $H\leq H'$, \cref{prop:morphism_recurrence} applied to the injection map tells us that any set of $H$-recurrence ($H$-Bohr recurrence) is a set of $H'$-recurrence (resp $H'$-Bohr recurrence). We state a converse statement for the Bohr case under a finite index condition.

\begin{prop}\label{prop:finite_index}
    Let $H\leq H'$ be locally compact abelian groups and assume that $H$ is closed and of finite index in $H'$. If $R\subseteq H'$ is a set of Bohr-$H'$-recurrence, then $R\cap H$ is a set of $H$-Bohr recurrence.   
\end{prop}

\begin{proof}
It suffices to show that for any $\chi_1,\ldots,\chi_k\in \widehat{H}$ and $\epsilon>0$, there exists $g\in R\cap H$ such that $g\in Bohr(\chi_1,\ldots,\chi_k; \epsilon)$. As $H$ is closed, we may extend $\chi_1,\ldots,\chi_k$ to characters in $\widehat{H'}$ (see for instance \cite[Theorem 2.1.4]{Rudin_Fourier_analysis_groups:1962}). Consider $((\mathbb{S}^1)^k,H')$ the rotation by these characters in $(\mathbb{S}^1)^k$, that is, $h$ maps $(x_1,\ldots,x_k)$ to $(\chi_1(h)x_1,\ldots, \chi_k(h)x_k)$. Consider the finite system $(H'/H,H')$, where the action is the natural one, and take $Y$ to be a minimal subsystem of $((\mathbb{S}^1)^k\times H'/H,H')$. Let $V\subseteq Y$ an open set with $V\subseteq U \times \{e\}$, where $U$ is an $\epsilon/2$-neighborhood of $1$ and $e$ is the identity in $H'/H$. As $R$ is a set of Bohr-$H'$ recurrence, there exists $g\in R$ such that $g^{-1}V\cap V \neq \emptyset$. This implies that $g\in R\cap H$ and $g\in Bohr(\chi_1,\ldots,\chi_k; \epsilon)$, as desired.

\end{proof}

\section{Bohr Recurrence for $\Z^d$-actions}\label{Sec4}
In this section we establish several properties of sets of $\Z^d$-Bohr recurrence, introducing the notion of Bohr correlations.

\subsection{Basic properties of sets of $\Z^d$-Bohr recurrence}
For $d\in \N$, we denote by $[d]=\{1,\ldots,d\}$ and for $x\in \R$ we denote the maximum integer $z$ such that $z\leq x$ by $\lfloor x\rfloor$, the decimal part of $x$ by $\{x\}=x-\lfloor x\rfloor\in [0,1)$, and the torus norm of $x$ by $\norm{x}_{\T}=\min\{\{x\}, 1- \{x\}  \}$, which is the distance of $x$ to the nearest
integer. We generalize this notation for $\boldsymbol{x}\in \R^r$ as follows:
$$ \lfloor \boldsymbol{x}\rfloor=(\lfloor x_i\rfloor)_{i\in [r]}, ~~ \{\boldsymbol{x}\}=(x_i-\lfloor x_i\rfloor)_{i\in [r]} \in [0,1)^r,~~\text{and} ~~\norm{\boldsymbol{x}}_{\T^r}=\sum_{i=1}^r \norm{x_i}_{\T}.$$
When it is clear, we will simply use $\norm{\boldsymbol{x}}$ instead of $\norm{\boldsymbol{x}}_{\T^r}$. With this notation we observe that $V\subseteq \Z^r$ is a Bohr$_0$ set if there exist $\epsilon>0$, $d\in \N$, and $\boldsymbol{\alpha}_1,\ldots,\boldsymbol{\alpha}_d\in \T^r$ such that
$$\{\boldsymbol{n}=(n_1,\ldots,n_r)\in \Z^r : \norm{n_i\boldsymbol{\alpha}_i}<\epsilon,\forall i\in[d]\} \subseteq V.$$

\begin{defn}[Ramsey property]
A property of subsets of $\Z^d$ is Ramsey if for any set $R\subseteq \Z^d$ that has this property and any partition $R=A\cup B$, at least one of $A$ or $B$ has this property.
\end{defn}
In other words, a family possesses the Ramsey property if and only if it is partition regular. As we pointed out in \cref{2.1} (right after \cref{Dual-obs}), the family of sets of $H$-recurrence is partition regular for each locally compact abelian group $H$. Consequently, we have the following proposition regarding the family of sets of $\Z^d$-Bohr recurrence.

\begin{prop}
The sets of $\Z^d$-Bohr recurrence have the Ramsey property.
\end{prop}
Analogously to the one dimensional case where sets of recurrence are usually considered in $\N$ or $\Z\setminus\{0\}$ to avoid $0$, we introduce the following definition. 
\begin{defn}[Essentiality]\label{Essential-set-of-Bohr-recurrence}
Let $R\subseteq \Z^d$ be a set of $\Z^d$-Bohr recurrence. We say that $R$ is \textit{essential} if $\forall \boldsymbol{n}\in R$, $\forall j\in [d]$, $n_j\neq 0$.
\end{defn}
From now one we make the abuse of notation denoting $\pi:\Z^d\to \Z^d$ the permutation of coordinates induced by a permutation $\pi:[d]\to [d]$.

The name ``essential'' is just to emphasize that this is the only relevant study case. In fact, the next proposition shows that every relevant set of $\Z^d$-Bohr recurrence can be reduced to an essential set of $\Z^d$-Bohr recurrence.
\begin{prop}\label{essential-Bohr-set}
    Let $R\subseteq \Z^d\setminus\{0\}$ be a set of $\Z^d$-Bohr recurrence. Then, there exist $d'\leq d$, a permutation $\pi\colon [d] \to [d]$, and an essential set of $\Z^{d'}$-Bohr recurrence $R'\subseteq \Z^{d'}$ such that $R'\times \{0\}^{d-d'}\subseteq \pi(R)$.
\end{prop} 
\begin{proof}
For $J\subseteq [d]$ denote 
$R_J= \{  (n_j)_{j\in [d]} \in R  :  n_j\neq 0 \iff j\in J\}.$ Notice that, as $0\notin R$, we have that 
$$ R= \bigcup_{\substack{J\subseteq [d]\\
J\neq \emptyset}} R_J.$$
By the Ramsey property, there exists $J\subseteq [d]$ nonempty such that $ R_J$ is a set of Bohr $\Z^{d}$-recurrence. Now, consider a permutation $\pi\colon \Z^d \to \Z^d$ taking the coordinates $J$ to $\{1,\ldots, d'\}$ maintaining their order, where $d'=|J|$. If we set $R'=p(R_J)$, where $p\colon\Z^d \to \Z^{J}$ is the canonical projection, then $R'\times \{0\}^{d-d'}\subseteq \pi(R)$ and $R'$ is an essential set of $\Z^{d'}$-Bohr recurrence.
\end{proof}
In particular, we can always reduce to the case in which $R$ is essential when studying recurrence in the family of nilsystems.
\begin{cor}\label{essential-sufficiency}
    Every set of $\Z^d$-Bohr recurrence is a set of recurrence for the family of $\Z^d$-nilsystems if and only if every essential set of $\Z^d$-Bohr recurrence is a set of recurrence for the family of $\Z^d$-nilsystems.
\end{cor}
\begin{proof}
We prove the nontrivial direction. Let $\mathscr{X}=(X,T_1,\ldots,T_d)$ be a minimal nilsystem and $R\subseteq \Z^d$ be a set of $\Z^{d}$-Bohr recurrence. We may assume $0\notin R$, otherwise the result is immediate. By \cref{essential-Bohr-set} there exist $d'\leq d$, a permutation $\pi\colon [d]\to [d]$ and an essential set of $\Z^{d'}$-Bohr recurrence $R'\subseteq \Z^{d'}$ such that $R'\times \{0\}^{d-d'}\subseteq \pi(R)$. Consider the minimal nilsystem
            $$\mathscr{Y}=(\overline{\mathcal{O}(e_X)}, T_{\pi(1)},\ldots, T_{\pi(d')}).$$ 
     Notice that for all $n\in R'$ if $m\in \{n\}\times \{0\}^{d-d'}\subseteq \pi(R)$ we have that
        $$ T_{\pi(1)}^{n_1}\cdots T_{\pi(d')}^{n_{d'}}= T_1^{m_{\pi^{-1}(1)}}\cdots T_d^{m_{\pi^{-1}(d)}}.$$
        In this way, since $R'$ is a set of recurrence for $\mathscr{Y}$, it follows that $R$ is a set of recurrence for $\mathscr{X}$.    
\end{proof}
The following property enables us to consider coordinates of an essential set of $\Z^{d}$-Bohr recurrence with arbitrary magnitudes. This is the intuitive generalization of the fact for sets of $\Z-$Bohr recurrence we can always eliminate a finite amount of elements without losing the property of Bohr recurrence.
\begin{prop}[Bands Property]\label{Bands-Prop}
Let $R\subseteq \Z^d$ be a set of $\Z^{d}$-Bohr recurrence, $k\in \Z\setminus\{0\}$ and $i\in [d]$. Set $B_k^i=\{\boldsymbol{n}\in \Z^d  :  n_i=k\}$, then the set $R_0=R\setminus B_k^i$ is a set of $\Z^{d}$-Bohr recurrence.
\end{prop}

\begin{proof}
We may assume $i=1$, since all the other cases are analogous.  It suffices to show that $\{k\}\times \Z^{d-1}$ is not a set of recurrence. Consider the system $(X=\Z_{|k|+1},T_1,\ldots,T_d)$, where $\Z_{ |k|+1}$ denotes the cyclic group with $(|k|+1)$ elements, $T_1(x)=x+1 \text{ mod } (|k|+1)$, and $T_j={\rm id}$ for $j\geq 2$. It is immediate to check that $\{k\}\times \Z^{d-1}$ is not a set of recurrence for this system. The Ramsey property allows us to conclude that $R_0$ is a set of recurrence.
\end{proof}
\begin{obs}
 We remark that in the previous proof one can also construct systems where all the transformations act in a non-trivial way.   
\end{obs}

Next, we prove that the property of being a set of $\Z^d$-Bohr recurrence is invariant under isomorphisms of $\Z^d$.
\begin{prop}\label{prop3.4}
Let $R\subseteq \Z^d$ be a set of $\Z^{d}$-Bohr recurrence and $M$ an invertible matrix with rational coefficients of dimension $d$. Then, the set defined by 
$$M^{-1}R\cap \Z^d=\{ n=(n_1,\ldots, n_d)^T\in  \Z^{d}  :  M n \in R \}, $$
is a set of $\Z^{d}$-Bohr recurrence. 
\end{prop}

\begin{proof}
By \cref{prop:morphism_recurrence}, the set $M^{-1}R$ is a set of $M^{-1}\Z^d$-recurrence. As $M$ has rational coefficients, \cref{prop:finite_index} implies that $M^{-1}\Z^d \cap \Z^d$ is of $\Z^d$-recurrence. 
\end{proof}

When dealing with sets of $\Z^d$-Bohr recurrence, it is common to encounter situations where some coordinates are rationally dependent. For instance, it is not difficult to demonstrate that if $R \subseteq \mathbb{Z}^2$ is a set of Bohr recurrence, then so is $R' = \{(n_1, n_2, n_1 + n_2) \in \mathbb{Z}^3 \mid (n_1, n_2) \in R\}$. 
 Naturally, the intrinsic recurrence behavior resides in the first two coordinates, rendering the third coordinate seemingly redundant. The following definition formalizes this concept, which will be important in the subsequent section for the purpose of eliminating certain cases.
\begin{defn}[Redundancy]
    Let $R\subseteq \Z^d$ be a set of $\Z^{d}$-Bohr recurrence. We say that $R$ is redundant if there is $d'<d$, a morphism $\phi:\Z^{d'}\to \Z^d$, and a set of $\Z^{d'}$-Bohr recurrence $R'$ such that $\phi(R')\subseteq R$. 
\end{defn}
\begin{obs} 
    Notice that $R\subseteq \Z^d$ is redundant if and only if there is $\boldsymbol{v}\in \Z^d \setminus \{0\}$ such that $R_{\boldsymbol{v}}=\{ \boldsymbol{n}\in R :  \boldsymbol{v}^T\cdot \boldsymbol{n}=0\}$ is a set of $\Z^{d}$-Bohr recurrence. In this case, $v$ is any non zero vector in $\Z^d$ orthogonal to the subspace $\phi(\Z^{d'})$, where $\phi:\Z^{d'}\to \Z^d$ is the morphism associated to $R\subseteq \Z^d$ being redundant. We also observe that taking $\boldsymbol{v}$ in $ \Q^d \setminus\{0\}$ instead of $\Z^d\setminus \{0\}$ leads to the same definition. 
\end{obs}
We can provide similar definitions in the context of locally compact abelian groups; however, for clarity, we will limit our discussion to cases involving $\Z^d$ or images of $\Z^d$ by matrices with rational coefficients. Note that in such cases, a morphism is just a matrix multiplication. 

To end this subsection, we prove that the property of being non-redundant is preserved under the transformation previously described in \cref{prop3.4}.
\begin{prop}\label{prop3.4-v2}
Let $R\subseteq \Z^d$ be a non redundant set of $\Z^{d}$-Bohr recurrence and $M$ be an invertible real matrix with rational coefficients of dimension $d$. Then, $M^{-1} R \cap \Z^d$ is non redundant as well.
\end{prop}
\begin{proof}
By contradiction, suppose that $M^{-1} R \cap \Z^d$ is redundant, then there exists $\boldsymbol{v}\in \Z^d\setminus\{0\}$ such that
  $$(M^{-1} R \cap \Z^d)_{\boldsymbol{v}}=\{\boldsymbol{n}\in M^{-1} R \cap \Z^d  :  \boldsymbol{v}^T\cdot \boldsymbol{n}=0\},$$
  is a set of $\Z^{d}$-Bohr recurrence. Hence, by \cref{prop3.4}, $M(M^{-1} R \cap \Z^d)_{\boldsymbol{v}}$ is a set of $\Z^{d}$-Bohr recurrence. If we define
  $$R_{M^{-T}\boldsymbol{v}}=\{ \boldsymbol{n}\in R  :  (M^{-T}\boldsymbol{v})^T\boldsymbol{n}=0\},  $$
  then $M(M^{-1} R \cap \Z^d)_{\boldsymbol{v}}\subseteq R_{M^{-T}\boldsymbol{v}}$ and therefore $R_{M^{-T}\boldsymbol{v}}$ is a set of $\Z^{d}$-Bohr recurrence. As $M^{-T}\boldsymbol{v}\in \Q^d\setminus\{0\}$, we have that $R$ is redundant, which is a contradiction. 
  \end{proof} 

\subsection{Bohr Correlations}
In this section we introduce the notion of Bohr correlation, which in simple words is a slope of the line around which a specific set of Bohr recurrence accumulates. To illustrate this notion first we provide some examples. 

\begin{ej}\label{ej1-Bohr}
If $R\subseteq \Z$ is a set of $\Z$-Bohr recurrence, then for every $\boldsymbol{\alpha}=(\alpha_1,\ldots,\alpha_d) \in \R^d$ we have that 
$$\tilde{R}=\{(\lfloor n\alpha_1 + 1/2 \rfloor ,\ldots, \lfloor n\alpha_d + 1/2 \rfloor )\in \Z^{d}  :  n\in R\}, $$
is a set of $\Z^{d}$-Bohr recurrence.
\end{ej}
In \cref{ej1-Bohr} the set of $\Z^d$-Bohr recurrence $\tilde{R}$ is approximately near the line of slope $(\alpha_1,\ldots,\alpha_d)$, in the sense that for $\boldsymbol{n}=(n_1,\ldots,n_d) \in \tilde{R}$ and $i,j\in [d]$ we have that $n_i/n_j\approx \alpha_i/\alpha_j$ as $\lim_{n\to \infty } \lfloor n\alpha_i + 1/2\rfloor/\lfloor n\alpha_j + 1/2\rfloor = \alpha_i/\alpha_j$.
\begin{proof}[Proof of \cref{ej1-Bohr}]
Let $R\subseteq \Z$ be a set of $\Z$-Bohr recurrence and consider $\alpha_1,\ldots,\alpha_d\in \R$. Define 
$$\tilde{R}=\{( \lfloor n\alpha_1 + 1/2\rfloor ,\ldots, \lfloor n\alpha_d + 1/2\rfloor )\in \Z^{d}  :  n\in R\}. $$
    Let $r\in \N,$ $\epsilon\in (0,1)$ and $\boldsymbol{\beta}_1,\ldots, \boldsymbol{\beta}_d\in [0,1)^r$. Consider the Bohr neighborhood of $0$
    $$B=\{n\in \Z  :  ~\norm{n\alpha_i }\leq \epsilon/2,~  \norm{n(\alpha_i\boldsymbol{\beta}_i) }\leq \epsilon/2,~ \forall i\in [d]\}.  $$ 
   
    As the family of Bohr neighborhood of $0$ is a filter, we have that $R\cap B$ is a set of $\Z$-Bohr recurrence. Taking $n\in R\cap B$ we have that for every $i\in [d]$
$$ \norm{\lfloor n\alpha_i + 1/2 \rfloor \boldsymbol{\beta}_i}= \norm{n\alpha_i\boldsymbol{\beta}_i + \boldsymbol{\beta}_i/2 - \{n\alpha+1/2\}\boldsymbol{\beta}_i} \\ \leq \norm{n\alpha_i\boldsymbol{\beta}_i}+ \norm{n\alpha_i}\cdot |\boldsymbol{\beta}_i|\leq \epsilon/2+ \epsilon/2=\epsilon.$$
Therefore $\tilde{R}$ is a set of $\Z^{d}$-Bohr recurrence.
\end{proof}

For the next example, we need the following proposition. 
\begin{prop}\label{prop4.4}
Let $R\subseteq \Z^2$. Denote $R_1=\{n\in \Z  :  \exists m\in \Z, (n,m)\in R\}$ and for $n\in R_1$ denote by $R(n,\bullet)=\{m\in\Z  :  (n,m)\in R\}.$
Suppose $R_1\subseteq \Z\setminus\{0\}$ is a set of $\Z$-Bohr recurrence. For $n\in R_1$ denote $L_n$ the length of the largest interval contained in $R(n,\bullet)$. If $L_n\to \infty$ as $n$ goes to infinity, then $R$ is a set of Bohr recurrence. 
\end{prop}
\begin{proof}
If $(X,g_1,g_2)$ is a minimal rotation, we can find an arbitrarily large $n_1\in R_1$ such that $d(g_1^{n_1}e_X,e_X)\leq \epsilon/2$. Observe that the set $N_{g_2}(e_X,B(e_X,\epsilon/2))$ is syndetic (i.e., it has bounded gaps) by the fact that the system $(X,g_2)$ is distal, so $(\overline{\mathcal{O}_{g_2}(e_X)},g_2)$ is minimal. Taking $n_1$ such that $L_{n_1}>L$ (with $L$ the syndetic constant associated to $B(e_X,\epsilon/2)$), we have that there exists $n_2\in R(n_1,\bullet)$ such that $d(g_2^{n_2}e_X,e_X)\leq \epsilon/2$. Thus, there exists $(n_1,n_2)\in R$ such that $d(g_1^{n_1}g_2^{n_2}e_X,e_X)\leq \epsilon$.
\end{proof}
\begin{ej}\label{Slope-example}
          From \cref{prop4.4}, the set $R=\{(n_1,n_2)\in \N^2  :  n_1^2\leq n_2\leq 2 n_1^2\}$ is a set of Bohr recurrence  such that is approximately near a line of slope $\infty$, in the sense that for every $\epsilon>0$, there is $M>0$ such that for all $(n_1,n_2)\in R\cap B_{\R^2}(0,M)^c$ we have that $|n_1/n_2|<\epsilon.$
\end{ej}

To formalize the notion illustrated in the examples, we first show that when studying $\Z^d$-Bohr recurrence, we can assume, without loss of generality, that the coordinates of the set of $\Z^d$-Bohr recurrence being studied are ordered. Although this assumption is not strictly necessary, it significantly simplifies the notation, making it easier to manipulate.

We call $R\subseteq \Z^d$ an \textit{ordered set of $\Z^d$-Bohr recurrence} if $R$ is a set of $\Z^d$-Bohr recurrence such that
$$\forall \boldsymbol{n}=(n_1,\ldots,n_d)\in R, ~~ |n_1|\geq \cdots\geq |n_d|.$$ 

\begin{obs}\label{Borh-ordered}
Let $R\subseteq \Z^d$ be a set of $\Z^d$-Bohr recurrence and $(X,T_1,\ldots,T_d)$ a system. Consider $S_d$ the symmetric group and notice that
 $$R=\bigcup_{\varphi\in S_d} R\cap \{(n_1,\ldots,n_d)\in \Z^d  :  |n_{\varphi(1)}|\geq \cdots \geq |n_{\varphi(d)}|\}.$$
By the Ramsey property, there is $\varphi \in S_d$ such that
$$R'= R\cap \{(n_1,\ldots,n_d)\in \Z^d  :  |n_{\varphi(1)}|\geq \cdots \geq |n_{\varphi(d)}|\},$$ 
is a set of $\Z^d$-Bohr recurrence. Note that recurrence in $(X,T_1,\ldots,T_d)$ with $R'$ is equivalent to recurrence in $(X,T_{\varphi(1)},\ldots,T_{\varphi(d)})$ with 
 $$R''=\{ (n_{\varphi(1)},\ldots ,n_{\varphi(d)})   :  (n_1,\ldots,n_d) \in R'\},$$
which is an ordered set of $\Z^d$-Bohr recurrence. Additionally, notice that if $R$ is essential, then so is $R''$.
\end{obs}

\begin{defn}[Bohr Correlations]\label{Bohr-Correlations}
Let $R\subseteq \Z^d$ be an essential and ordered set of $\Z^d$-Bohr recurrence and $I\subseteq \{(i,j)\in [d]^2  :  j\geq i\}$. We define a vector of $I$-Bohr correlations for $R$ as $\boldsymbol{P}=(P_{i,j})_{j\geq i}\in [-1,1]^{d\times (d+1)/2}$ such that $\forall \epsilon>0$,
$$R_{\boldsymbol{P},I,\epsilon} =\Bigl\{\boldsymbol{n}\in R  :  \left |\frac{n_{j}}{n_i}-P_{i,j} \right |\leq \epsilon,~ \forall (i,j)\in I\Bigr\}$$
is a set of $\Z^d$-Bohr recurrence. 
In the case $I= \{(i,j)\in [d]^2  :  j\geq i\}$ we drop the prefix $I$ and say that $\boldsymbol{P}$ is a vector of Bohr correlations. We denote by $\mathcal{BC}(R)$ the set of Bohr correlations of $R$.
\end{defn}
For example, for $d=2$, a vector of Bohr correlations is just the slope of a line for which many points (a set of Bohr recurrence amount of points) are inside a cone of angle $\epsilon$ around such line. Notice that it is not always possible to take $\epsilon=0$ (for instance, this is not possible in the set of recurrence from \cref{Slope-example}).

The existence of Bohr correlations is not immediate from the definition, nor is the possibility of extending a vector of $I$-Bohr correlations to a vector of Bohr correlations. We address both issues in the following proposition.

\begin{prop}\label{Bohrcompleteness}
Let $R\subseteq \Z^d$ be an essential and ordered set of $\Z^d$-Bohr recurrence. Consider $I\subseteq \{(i,j)\in [d]^2  :  j\geq i\}$ and $\boldsymbol{Q}=(Q_{i,j})_{(i,j)\in I}\in [-1,1]^I$ a vector of $I$-Bohr correlations. Then, there exists $\boldsymbol{P}\in \mathcal{BC}(R)$ that extends $(Q_{i,j})_{(i,j)\in I}$, in the sense that $P_{i,j}=Q_{i,j}$, $\forall (i,j)\in I$. In particular, for $I=\emptyset$ we have that $\mathcal{BC}(R)\neq \emptyset$.
\end{prop}
\begin{proof}
 Notice that $(\frac{n_{j}}{n_i})_{j\geq i}\in [-1,1]^{d(d+1)/2}$, $\forall \boldsymbol{n}\in R$. For $m\in \N$, we can cover $[-1,1]^{d(d+1)/2}$ with finitely many intervals of the form $\prod_{j\geq i}
 (P_{i,j}-\frac{1}{m},P_{i,j}+\frac{1}{m})$, for $(P_{i,j})_{j\geq i}\in ([-1,1]\cap \Q)^{d(d+1)/2}$. For every $\boldsymbol{n}\in R_{\boldsymbol{Q},I,1/m}$ we have that $(\frac{n_{j}}{n_i})_{j\geq i} \subseteq [-1,1]^{d(d+1)/2}$. Hence, using the Ramsey property (after turning the aforementioned covering into a partition), we obtain $\boldsymbol{P}^{m}=(P_{i,j}^{m})_{j\geq i}\in ([-1,1]\cap \Q)^{d(d+1)/2}$ such that $R_{\boldsymbol{Q},I,1/m}\cap R_{\boldsymbol{P^m},1/m}$ is a set of $\Z^d$-Bohr recurrence, where we define for each $m\in \N$
$$R_{\boldsymbol{P}^m,1/m}:=\Bigl\{\boldsymbol{n}\in R~:~ \Bigl|\frac{n_{j}}{n_i}-P_{i,j}^{m}\Bigr|< \frac{1}{m} ,~\forall j\geq i\Bigr\}.$$
Passing to a subsequence, we may assume that $(\boldsymbol{P}^{m_l})_{l\geq 1}$ converges to $\boldsymbol{P}\in [-1,1]^{d(d+1)/2}$. We claim that 
$$R_{\boldsymbol{P},\epsilon}:=\Bigr\{\boldsymbol{n}\in R ~:~  \Bigl|\frac{n_{j}}{n_i}-P_{i,j}\Bigl|< \epsilon, ~  \forall j\geq i\Bigl\}, $$
is a set of $\Z^d$-Bohr recurrence for all $\epsilon>0$ and that $\boldsymbol{P}$ coincides with $\boldsymbol{Q}$ in $I$. 
Indeed, for the first part notice that as $\boldsymbol{P}^{m_l}\to \boldsymbol{P}$ when $l\to \infty$, we can take $l$ large enough so that $(P^{m_l}_{i,j}-1/m_l,P^{m_l}_{i,j}+ 1/m_l)\subseteq (P_{i,j}-\epsilon,P_{i,j}+\epsilon), \forall j\geq i.$
Therefore, $ R_{\boldsymbol{P}^{m_l},1/m_l}\subseteq R_{\boldsymbol{P},\epsilon} $ and since $R_{\boldsymbol{P}^{m_l},1/m_l}$ is a set of $\Z^d$-Bohr recurrence, so is $R_{\boldsymbol{P},\epsilon}$. We conclude that $\boldsymbol{P}\in \mathcal{BC}(R)$. 
Now for the second part, let $\epsilon>0$ and take $l\in \N$ large enough so that $|P_{i,j}^{m_l}-P_{i,j}|\leq \epsilon$, $\forall j\geq i$, and such that $\frac{1}{m_l}\leq \epsilon$. For all $n\in R_{\boldsymbol{Q},I,1/m_l}\cap R_{\boldsymbol{P}^{m_l},1/m_l}$, and $(i,j)\in I$,
$$|P_{i,j}- Q_{i,j}|\leq \left |\frac{n_j}{n_i}- P_{i,j}^{m_l} \right |  + \left |P_{i,j}^{m_l}- P_{i,j} \right |  +\left | \frac{n_j}{n_i}- Q_{i,j} \right |\leq \frac{3}{m_l}\leq 3\epsilon. $$
Taking $\epsilon\to 0$ yields $P_{i,j}=Q_{i,j}$, $\forall (i,j)\in I$, concluding.
\end{proof}

Now we show that Bohr correlations are consistent, in the sense that if $1\leq i\leq j \leq l\leq d$, then the correlation between the $i$-th coordinate and the $l$-th coordinate is the product of the correlations between the $i$-th coordinate and the $j$-th coordinate, and the correlation between the $j$-th coordinate and the $l$-th 
 coordinate.

\begin{lema}\label{consistentency}
Let $R\subseteq \Z^d$ be an essential and ordered set of $\Z^d$-Bohr recurrence, and $\boldsymbol{P}\in \mathcal{BC}(R)$. Then, for all $ i,j,l\in [d]$ with $i\leq j\leq l$, 
$P_{i,l}=P_{i,j}P_{j,l}. $
\end{lema}
\begin{proof}
Let $\epsilon>0$, we have that $\forall \boldsymbol{n}\in R_{\boldsymbol{P},\epsilon}$
\begin{align*}
 |P_{i,l}-P_{i,j}P_{j,l}|&\leq \left |\frac{n_l}{n_i}- P_{i,l}\right | + \left | \frac{n_j}{n_i}\frac{n_l}{n_j}-P_{i,j}P_{j,l}\right|   \\
 &\leq \epsilon + \left| \frac{n_j}{n_i}\frac{n_l}{n_j}-P_{i,j}\frac{n_l}{n_j} \right|+ \left| P_{i,j}\frac{n_l}{n_j}  - P_{i,j}P_{j,l}\right| \\
 &\leq   \epsilon + \left | \frac{n_l}{n_j}\right |\epsilon+ \left |P_{i,j} \right | \epsilon \leq 3 \epsilon. 
\end{align*}
Letting $\epsilon$ tend to $0$ we get that $P_{i,l}=P_{i,j}P_{j,l}$.
\end{proof}

Now, we define the property of complete independence, which comprises the case where the correlations are not rational.

\begin{defn}\label{Complete-independence}
    Let $R\subseteq \Z^d $ be an essential and ordered set of $\Z^d$-Bohr recurrence. We say that $R$ has the property of complete independence if there is $\boldsymbol{P}\in \mathcal{BC}(R)$ such that for all $i\in[d]$, the set
    $\{ P_{i,j}  :  j\geq i, P_{i,j}\neq 0\} $ is rationally independent.
\end{defn}

The main consequence of the property of complete independence is that it allows us to approximate multiple coordinates with just one variable, which is the statement of \cref{Theorem-B}. For the readers' convenience, we state \cref{Theorem-B} again below.

\begin{customthm}{B}

Let $R$ be an essential and ordered set of $\Z^d$-Bohr recurrence with the property of complete independence. Then, for all $r\in \N$ and $\epsilon>0$ there exists a set of Bohr recurrence $R_\epsilon\subseteq R$ such that for any $\boldsymbol{n}\in R_\epsilon$ and $\boldsymbol{w}_1,\ldots,\boldsymbol{w}_d\in \T^r$ there is $y\in \R^r$ such that
\begin{align*}
||y||_{\R^r}<\epsilon ~~ \text{ and }~~ ||(n_1y,\ldots,n_d y)-(\boldsymbol{w}_1,\ldots,\boldsymbol{w}_d)||_{\T^{d\cdot r}}<\epsilon.
\end{align*}
\end{customthm}

\begin{proof}
Let $\boldsymbol{P}\in \mathcal{BC}(R)$ such that for all $i\in[d]$,
    $\{ P_{i,j}  :  j\geq i, P_{i,j}\neq 0\} $ are rationally independent. Let $\epsilon>0$ and for $l \in [d]$ denote $I_l=\{ j \in [d]  :  j> l, P_{l,j}\neq 0\}$. As $\{ P_{l,j}\}_{j\in I_l\cup\{l\}}$ is rationally independent, we can find $N\geq r$ large enough such that for all $l\in [d]$, $\{ (nP_{l,j})_{j\in I_l}  :  n\in [-N,N]\}$, is $\epsilon$-dense in $\T^{|I_l|}$. 

Let $M>N/\epsilon$. Take $\boldsymbol{n}\in R_{\boldsymbol{P},\epsilon/M}\cap B(0,M)^c$ and $\boldsymbol{w}_1,\ldots,\boldsymbol{w}_d\in\T^r$. We will assume without loss of generality that $\boldsymbol{w}_1,\ldots,\boldsymbol{w}_d\in [0,1)^r$. Define $\varphi:\R^r\to \R^{d\cdot r}$ by $\varphi(\boldsymbol{x})=(n_1\boldsymbol{x},\ldots,n_d\boldsymbol{x})$ and consider $\boldsymbol{y}=\sum_{i=1}^d\frac{\boldsymbol{y}_i}{n_i},$ where we define $\{\boldsymbol{y}_i\}_{i=1}^d$ inductively as follows: first we define $\boldsymbol{y}_d=\boldsymbol{w}_d, $ and note that $||\frac{\boldsymbol{y}_d}{n_d}||_{\R^r}\leq |\frac{1}{n_d}|\leq \epsilon.$ Next, suppose we have defined for $1\leq l<d$: $\boldsymbol{y}_{l+1},\ldots,\boldsymbol{y}_d \in \R^r\cap B(0,C)$, where $C=N+d+1$. Take $\boldsymbol{k}_{l}\in [-N,N]^{r}$ such that for all $j> l$
    \begin{equation}\label{eq4.6}
        ||P_{l,j}\boldsymbol{k}_{l}+ P_{l,j}(\boldsymbol{w}_{l}-\sum_{i=l+1}^d\{\frac{n_l}{n_i}\boldsymbol{y}_{i}\}) ||_{\T^r}\leq\epsilon.
    \end{equation}
    The existence of each coordinate of $\boldsymbol{k}_l$ is guaranteed by the definition of $N$ and the fact that such condition is trivial for $j\notin I_l$. Define 
    $$\boldsymbol{y}_l=\boldsymbol{w}_l +\boldsymbol{k}_l - \sum_{i=l+1}^d\{\frac{n_l}{n_i}\boldsymbol{y}_{i}\} \in \R^r\cap B(0,C). $$

 In this way, we have that $\norm{\boldsymbol{y}_l/n_l}_{\R^r}\leq  C/|n_l|\leq  (1+\frac{d+1}{N}) \epsilon \leq (d+2)\epsilon.$ For each $j\in [d]$ denote $\boldsymbol{\mathfrak{e}}_j\in (\R^{r\cdot r})^d$ the vector which $j$-th coordinate is the identity matrix on $\R^r$ and each other coordinate is a matrix full of zeros. In other words, we have that for each $\boldsymbol{x}_1,\ldots,\boldsymbol{x}_d\in (\R^r)^d$,
 $$(\boldsymbol{x}_1,\ldots,\boldsymbol{x}_d)= \sum_{j=1}^d\boldsymbol{\mathfrak{e}}_j\boldsymbol{x}_j. $$
 We claim that $$\norm{\varphi(y) -w}_{\T^{d\cdot r}}=\norm{\sum_{i=1}^d 
        \sum_{j>i}\frac{n_j}{n_i}\boldsymbol{\mathfrak{e}}_j\boldsymbol{y}_i}_{\T^{d\cdot r}}.$$

For this we will prove by induction that for every $l=0,...,d$ we have that  
        \begin{align*}
         \norm{\varphi(y) -w}_{\T^{d\cdot r}}= \norm{\sum_{j=1}^{d}n_j\boldsymbol{\mathfrak{e}}_j(\sum_{i=1}^l\frac{\boldsymbol{y}_i}{n_i})+\sum_{i=l+1}^d 
        \sum_{j>i}\frac{n_j}{n_i}\boldsymbol{\mathfrak{e}}_j\boldsymbol{y}_i
        +\sum_{j=1}^{l}\sum_{i=l+1}^d \frac{n_j}{n_i}\boldsymbol{\mathfrak{e}}_j\boldsymbol{y}_i -\sum_{j=1}^l \boldsymbol{\mathfrak{e}}_j \boldsymbol{w}_j}_{\T^{d\cdot r}}.
    \end{align*}
Indeed, we start with $l=d$ for which by definition we have that 
    \begin{align*}
         \norm{\varphi(y) -w}_{\T^{d\cdot r}}=\norm{\sum_{j=1}^dn_j \boldsymbol{\mathfrak{e}}_j  
         (\sum_{i=1}^d\frac{\boldsymbol{y}_i}{n_i}) - \sum_{j=1}^d \boldsymbol{\mathfrak{e}}_j\boldsymbol{w}_j}_{\T^{d\cdot r}}.
    \end{align*}

Now assume that for $1\leq l \leq d$ we have proved that 
        \begin{align*}
         \norm{\varphi(y) -w}_{\T^{d\cdot r}}= \norm{\sum_{j=1}^{d}n_j\boldsymbol{\mathfrak{e}}_j(\sum_{i=1}^l\frac{\boldsymbol{y}_i}{n_i})+\sum_{i=l+1}^d 
        \sum_{j>i}\frac{n_j}{n_i}\boldsymbol{\mathfrak{e}}_j\boldsymbol{y}_i
        +\sum_{j=1}^{l}\sum_{i=l+1}^d \frac{n_j}{n_i}\boldsymbol{\mathfrak{e}}_j\boldsymbol{y}_i -\sum_{j=1}^l \boldsymbol{\mathfrak{e}}_j \boldsymbol{w}_j}_{\T^{d\cdot r}}.
    \end{align*}

Using the definition of $\boldsymbol{y}_l$ we obtain that
    \begin{align*}
&\norm{\sum_{j=1}^{d}n_j\boldsymbol{\mathfrak{e}}_j(\sum_{i=1}^l\frac{\boldsymbol{y}_i}{n_i})+\sum_{i=l+1}^d 
        \sum_{j>i}\frac{n_j}{n_i}\boldsymbol{\mathfrak{e}}_j\boldsymbol{y}_i
    +\sum_{j=1}^{l}\sum_{i=l+1}^d \frac{n_j}{n_i}\boldsymbol{\mathfrak{e}}_j\boldsymbol{y}_i -\sum_{j=1}^l\boldsymbol{\mathfrak{e}}_j \boldsymbol{w}_j}_{\T^{d\cdot r}}\\
         &= \norm{\sum_{j=1}^{d}n_j\boldsymbol{\mathfrak{e}}_j(\sum_{i=1}^{l-1}\frac{\boldsymbol{y}_i}{n_i})+\sum_{i=l+1}^d 
        \sum_{j>i}\frac{n_j}{n_i}\boldsymbol{\mathfrak{e}}_j\boldsymbol{y}_i +\sum_{j=1}^{l-1}\sum_{i=l+1}^d \frac{n_j}{n_i}\boldsymbol{\mathfrak{e}}_j\boldsymbol{y}_i -\sum_{j=1}^{l-1}\boldsymbol{\mathfrak{e}}_j \boldsymbol{w}_j + \sum_{\substack{j=1\\j\neq l}}^d \frac{n_j}{n_l}\boldsymbol{\mathfrak{e}}_j\boldsymbol{y}_l       }_{\T^{d\cdot r}}\\
         &= \norm{\sum_{j=1}^{d}n_j\boldsymbol{\mathfrak{e}}_j(\sum_{i=1}^{l-1}\frac{\boldsymbol{y}_i}{n_i})+\sum_{i=l}^d 
        \sum_{j>i}\frac{n_j}{n_i}\boldsymbol{\mathfrak{e}}_j\boldsymbol{y}_i+\sum_{j=1}^{l-1}\sum_{i=l}^d \frac{n_j}{n_i}\boldsymbol{\mathfrak{e}}_j\boldsymbol{y}_i -\sum_{j=1}^{l-1} \boldsymbol{\mathfrak{e}}_j\boldsymbol{w}_j}_{\T^{d\cdot r}},\\
    \end{align*}
    in where the last equality comes from splitting the last summation into $j>l$ and $j<l$, integrating the former terms into the second summation and the later terms into the third summation, namely: 
    $$ \sum_{i=l+1}^d   \sum_{j>i}\frac{n_j}{n_i}\boldsymbol{\mathfrak{e}}_j\boldsymbol{y}_i +\sum_{j=1}^{l-1}\sum_{i=l+1}^d \frac{n_j}{n_i}\boldsymbol{\mathfrak{e}}_j\boldsymbol{y}_i + \sum_{\substack{j=1\\j\neq l}}^d \frac{n_j}{n_l}\boldsymbol{\mathfrak{e}}_j\boldsymbol{y}_l= \sum_{i=l}^d 
        \sum_{j>i}\frac{n_j}{n_i}\boldsymbol{\mathfrak{e}}_j\boldsymbol{y}_i+\sum_{j=1}^{l-1}\sum_{i=l}^d \frac{n_j}{n_i}\boldsymbol{\mathfrak{e}}_j\boldsymbol{y}_i
      .$$
This concludes the induction, and using the case $l=0$ we get that 
  $$   \norm{\varphi(y) -w}_{\T^{d\cdot r}}=\norm{\sum_{i=1}^d 
        \sum_{j>i}\frac{n_j}{n_i}\boldsymbol{\mathfrak{e}}_j\boldsymbol{y}_i}_{\T^{d\cdot r}}.$$

Moreover, by \eqref{eq4.6} we have that for each $l\in [d]$
\begin{align*}
||P_{l,j}\boldsymbol{y}_l||_{\T^{ r}}&= \Big\Vert  P_{l,j}\boldsymbol{w}_{l}+P_{l,j}\boldsymbol{k}_{l}-P_{l,j}\sum_{i=l+1}^d\{\frac{n_l}{n_i}\boldsymbol{y}_{i}\}\Big\Vert_{\T^{ r}} \leq \epsilon .
\end{align*}

Therefore, we found $y\in B_{\R^r}(0,\epsilon d(d+2))$ such that
\begin{align*}
        \norm{\varphi(y) -w}_{\T^{d\cdot r}}=\norm{\sum_{i=1}^d 
        \sum_{j>i}\frac{n_j}{n_i}\boldsymbol{\mathfrak{e}}_j\boldsymbol{y}_i}_{\T^{d\cdot r}}&\leq \norm{\sum_{i=1}^d \sum_{j>i}P_{i,j}\boldsymbol{\mathfrak{e}}_j\boldsymbol{y}_i}_{\T^{d\cdot r}} +  \sum_{i=1}^d\sum_{j>i} \norm{(\frac{n_j}{n_i}-P_{i,j})\boldsymbol{\mathfrak{e}}_j\boldsymbol{y}_i}_{\T^{d\cdot r}}  \\
        &\leq\sum_{i=1}^d \sum_{j>i} \norm{P_{i,j}\boldsymbol{y}_i}_{\T^{ r}} + \frac{\epsilon Cd^2}{M}\\
        &\leq d^2\epsilon+d^2\epsilon^2 (d+2)\leq (d+3)^3 \epsilon.
\end{align*}
We just proved that for each $r\in \N$ and $\epsilon>0$, there is $M\in \N$ such that for each $\boldsymbol{n}\in R_{P,\epsilon/M}\cap B(0,M)^c $ and $\boldsymbol{w}_1,\ldots,\boldsymbol{w}_d\in \T^r$ there is $\boldsymbol{y}\in \R^r$ satisfying
\begin{equation*}
    ||\boldsymbol{y}||_{\R^r}<d(d+2)\epsilon  ~~ \text{ and }~~ ||(n_1\boldsymbol{y},\ldots,n_d \boldsymbol{y})-(\boldsymbol{w}_1,\ldots,\boldsymbol{w}_d)||_{\T^{d\cdot r}}<(d+3)^3\epsilon.
\end{equation*}

Given that $\epsilon>0$ is arbitrary in the previous statement, we can replace $\epsilon$ with $\epsilon/(d+3)^3$ getting $M\in \N$ such that for each $\boldsymbol{n}\in R_\epsilon:= R_{P,\epsilon/((d+3)^3M)}\cap B(0,M)^c $ and $\boldsymbol{w}_1,\ldots,\boldsymbol{w}_d\in \T^r$ there is $y\in \R^r$ such that 
\begin{equation*}
    ||\boldsymbol{y}||_{\R^r}<d(d+2)\epsilon/(d+3)^3< \epsilon  ~~ \text{ and }~~ ||(n_1\boldsymbol{y},\ldots,n_d \boldsymbol{y})-(\boldsymbol{w}_1,\ldots,\boldsymbol{w}_d)||_{\T^{d\cdot r}}<\epsilon,
\end{equation*}
    concluding.
\end{proof} 
\section{Recurrence in $\Z^d$-Weyl systems}\label{Sec5}

In this section, our objective is to lift $\Z^d$-Bohr recurrence to the broader class of $\Z^d$-Weyl systems. First, we begin with a brief outline of the method. Then, we introduce the concept of total ergodicity for actions of $\Z^d$-groups. Notably, this property is characterized within nilsystems through connectedness. Subsequently, we demonstrate that when extending Bohr recurrence to nilsystems, we can consistently simplify the problem by focusing on scenarios where the set of Bohr recurrence has the property of complete independence. Finally, we use this reduction to establish that sets characterized by $\Z^d$-Bohr recurrence are sets of recurrence for $\Z^d$-Weyl systems.

\subsection{Pivotal Lemma}\label{5.2}
In order to establish \cref{Theorem-A}, \cref{Theorem-B} gives as a corollary a pivotal technical lemma that comes into play. Specifically, when dealing with a set of Bohr recurrence with the property of complete independence, this lemma facilitates the approximation of the final commutator group within a nilspace. This approximation is achieved by applying the powers of commutators from the dynamics, all of which are drawn from the aforementioned set of Bohr recurrence, to a single variable.

\begin{lema}\label{lema-previo-resultado-principal}
Let $s\geq 2$ and let $(X= G/\Gamma,T_1,\ldots,T_d)$ be a minimal $s$-step $\Z^d$-affine system, and $R$ an essential and ordered set of $\Z^d$-Bohr recurrence with the property of complete independence. Then, for all $\epsilon>0$ there is a set of $\Z^d$-Bohr recurrence $R_\epsilon\subseteq R$ such that for all $v\in G_{s}$ and all $\boldsymbol{n}\in R_\epsilon$ there exist $h\in G_{s-1}$ such that 
\begin{equation}
  d_G(h,e_G)< \epsilon, \text{  and  } d_X([h,\tau_1^{n_1}]\cdots [h,\tau_d^{n_d}]\Gamma,v \Gamma)<\epsilon.
\end{equation}
\end{lema}
\begin{proof}
Writing $X=\T^p$ and $T_ix=\boldsymbol{M}_ix+\boldsymbol{\alpha}_i$ where $\boldsymbol{M}_1,\ldots,\boldsymbol{M}_d$ are unipotent matrices and $\boldsymbol{\alpha}_1,\ldots,\boldsymbol{\alpha}_d\in \T^p$, we recall from \cref{Weyl-section} that we consider $\Gamma$ the group (with composition) generated by the transformations given by the matrices $\boldsymbol{M}_1,\ldots,\boldsymbol{M}_d$; and $G$ corresponds to the group generated by $\Gamma$ and the transformations obtained by translations of elements in $\T^p$. Moreover, we have that $G_0$ is isomorphic to $\T^p$ and thus $G=G_0\Gamma$.

We will use vector notation $\boldsymbol{g}\in \T^p$ for the representatives of elements $g\in G$ that correspond to the rotation by $\boldsymbol{g}$.
We use multiplicative notation for the operation on $G$, but notice that such operation reduces to addition for the representatives in $\T^p$ of elements in $G$. We recall that by \cref{Parry-theorems} the group $G_j$ can be seen as a subgroup of $\T^p$ for each $j\geq 2$. Moreover, for $j\geq 2$
\begin{equation}\label{eq-commutators}
  G_j\cong \left\{  \sum_{i=1}^d (\boldsymbol{M}_i-\boldsymbol{I})\boldsymbol{g}_i \mid \boldsymbol{g}_i\in G_{j-1}\cap G_0, \text{ for each }i\in [d] \right\} \subseteq \T^p,  
\end{equation}
in where the intersection with $G_0$ only plays a role when $j=2$, because in such case $G_{j-1}=G$, and $G_{j-1}\subseteq G_0$ otherwise.

Denote by $\Psi\colon (G_{s-1}/ G_{s-1} \cap\Gamma)^d  \to G_s $ the map defined by 
$$\Psi(g_1\Gamma,\ldots,g_d\Gamma)=[g_1,\tau_1]\cdots [g_d, \tau_d], ~ \text{ for } g_1,\ldots,g_d\in G_{s-1}\cap G_0,$$ 
which is well defined due to the fact that $G_0\cong \T^p$, and therefore from the definition of $X$ we get $X=G_0\Gamma/\Gamma$ with $G_0\cap \Gamma=\{e_G\}$. 
 Following the comments in \cref{Weyl-section},  for each $i\in [d]$ the commutator $[g_i,\tau_i]$ corresponds to the element $(\boldsymbol{M}_{g_i}-\boldsymbol{I})\boldsymbol{\alpha}_i -(\boldsymbol{M}_i-\boldsymbol{I})\boldsymbol{g}_i$, but $g_i\in  G_0\cong \T^p$ yields $\boldsymbol{M}_{\boldsymbol{g}_i}=\boldsymbol{I}$. In consequence, the commutator $[g_i,\tau_i]$ is represented by the element $-(\boldsymbol{M}_i-\boldsymbol{I})\boldsymbol{g}_i\in \T^p$.

 Observe that there is $q\in \N$ such that $(G_{s-1}/ G_{s-1} \cap\Gamma)\cong \T^q$. Indeed, this is direct from \cref{Parry-theorems} if $s\geq 3$ as $(G_{s-1}/ G_{s-1} \cap\Gamma)=G_{s-1}$ is a subtorus of $\T^p$, and in the case $s=2$, we have that $(G_{s-1}/ G_{s-1} \cap\Gamma)=G/\Gamma=\T^p$ so the same applies for $q=p$. Let $\Phi: \T^q \cong (G_{s-1}/ G_{s-1} \cap\Gamma)\to \T^p$ be such an embedding. Thus, the function $\Psi$ seen as a function from $(\T^q)^d \to \T^p$ corresponds to 
$$\Psi(\boldsymbol{g}_1,\ldots,\boldsymbol{g}_d) =-\sum_{i=1}^d (\boldsymbol{M}_i-\boldsymbol{I})\Phi(\boldsymbol{g}_i), ~ \text{ for } \boldsymbol{g}_1,\ldots,\boldsymbol{g}_d\in \T^q  ,$$

The linear map $\Psi\colon(\T^q)^{d}\to \T^p$ is a uniformly continuous morphism. Let $\epsilon>0$ and take $\delta\in (0,\epsilon)$ be a constant of uniform continuity for $\Psi$ for $\epsilon$. Using \cref{Theorem-B} with $\delta$, we find $R_\epsilon\subseteq R$ such that for every $\boldsymbol{h}_1,\ldots,\boldsymbol{h}_d\in \T^q$ and $\boldsymbol{n}\in R_\epsilon$ there exists $\boldsymbol{h}\in \R^q\cap B(0,\epsilon)$ satisfying
$$\norm{(n_1\boldsymbol{h},\ldots,n_d\boldsymbol{h}) - (\boldsymbol{h}_1,\ldots,\boldsymbol{h}_d)}_{\T^{d\cdot q}}<\delta. $$
Therefore, we get that 
$$\norm{\Psi(n_1\boldsymbol{h},\ldots,n_d \boldsymbol{h})-\Psi(\boldsymbol{h}_1,\ldots,\boldsymbol{h}_d)  }_{\T^p}<\epsilon,$$
or equivalently
$$d_{G_s}([h,\tau_1^{n_1}]\cdots [h,\tau_d^{n_d}],   [h_1,\tau_1]\cdots [h_d,\tau_d])<\epsilon. $$
In consequence
 $$d_{X}([h,\tau_1^{n_1}]\cdots [h,\tau_d^{n_d}]\Gamma,   [h_1,\tau_1]\cdots [h_d,\tau_d]\Gamma)<\epsilon.$$
 As every element of $G_s$ can be represented as $[h_1,\tau_1]\cdots[h_d,\tau_d]$ for some $h_1,\ldots,h_d\in G_{s-1}\cap G_0$ by \cref{eq-commutators}, we conclude.
\end{proof}
\cref{lema-previo-resultado-principal} allows us to prove \cref{Theorem-A} in the case where $R$ is an essential and ordered set of $\Z^d$-Bohr recurrence with the property of complete independence and $X$ is connected. In what follows, we will show how to reduce to such a case and finally show how it can be used to prove \cref{Theorem-A}.

\subsection{Connectedness in Nilsystems}
Notice that in $(X=G/\Gamma, T_1,\ldots,T_d)$ we can always assume that $G$ is spanned by $G_0$ and the elements $\tau_1,\ldots,\tau_d$ defining the dynamics. Indeed, setting $G'=\langle G_0,\tau_1,\ldots,\tau_d\rangle$, we have that $G'$ is an open subgroup of $G$, and since for all $a\in X$ the map $g\to g\cdot a$ is open, the sets $G'\cdot a$, $a\in X$, are open subsets of $X$ that are pairwise equal or disjoint. Since these sets cover $X$, they are closed in $X$, hence compact. Moreover, given that $X$ is compact, there exist $a_1,\ldots,a_k\in X$ such that $\{X_i=G'\cdot a_i\}_{i=1}^k$ covers $X$. For $i=1,\ldots,k$, let $\Gamma_i$ denote the stabilizer of $a_i$ in $G$, that is, $\Gamma_i=g_i\Gamma g_i^{-1}, $
where $g_i\in G$ is an element such that $g_i\cdot e_X=a_i$. Note that $\Gamma_i\cap G'$ represents the stabilizer of $a_i$ in $G'$, and $X_i$ can be viewed as the nilmanifold $G'/(\Gamma_i\cap G')$. Since $\tau_1,\ldots,\tau_d\in G'$, $X_i$ is a $\langle\tau_1,\ldots,\tau_d\rangle$-invariant set and $(X_i,T_1,\ldots,T_d)$ is a nilsystem. In this light, we obtain a partition of $X$ into finitely many nilsystems, and each one can be studied separately. Thus, without loss of generality, we can substitute $G$
by $G'$ and assume that $G$ is spanned by $G_0$ and $\tau_1,\ldots,\tau_d$. 

In order to characterize connectedness in nilsystems, we will first need the following result that characterizes the ergodicity in a $\Z^d$-torus rotation. 
\begin{teo}\label{teo3.1}
Let $(\T^N,\boldsymbol{\alpha}_1,\ldots,\boldsymbol{\alpha}_d)$ be a $\Z^d$-rotation. Then, the following are equivalent:
\begin{enumerate}
    \item $(\T^N,\boldsymbol{\alpha}_1,\ldots,\boldsymbol{\alpha}_d)$ is ergodic,
    \item $\forall k\in \Z^N\setminus\{0\}$, $\exists i\in [d]$ such that $k\cdot \boldsymbol{\alpha}_i\notin \Z$,
    \item $\forall k\in \Z^N\setminus\{0\}$, $\exists (t_i)_{i=1}^d\in \Z^d$, $k\cdot (\sum_{i=1}^d t_i\boldsymbol{\alpha}_i) \notin \Z$.
\end{enumerate}
\end{teo}

\begin{proof}
((1) $\Longrightarrow $ (2)) We prove the contrapositive assertion. Suppose that $\exists \boldsymbol{k}\in \Z^N\setminus\{0\}$ such that $\forall i\in [d]$ we have $\boldsymbol{k}\cdot \boldsymbol{\alpha}_i\in \Z$. Thus,  $f(x)=e^{2\pi i \boldsymbol{k}\cdot x}$ is an invariant function in $L^\infty(\T^N)$ that is not constant, and thus $(\T^N,\boldsymbol{\alpha}_1,\ldots,\boldsymbol{\alpha}_d)$ is not ergodic.

((2) $\Longrightarrow $ (3)) This implication is direct.

((3) $\Longrightarrow $ (1)) Suppose that $(\T^N,\boldsymbol{\alpha}_1,\ldots,\boldsymbol{\alpha}_d)$ is not ergodic. Then, there exists a $\langle \boldsymbol{\alpha}_1,\ldots,\boldsymbol{\alpha}_d\rangle$-invariant function $f\in L^\infty(\T^N)$ which is not constant. Writing $f$ in the classic basis of $L^2(\T^N)$
$$f(x)=\sum_{\boldsymbol{k}\in \Z^N} c_{\boldsymbol{k}} e^{2\pi i \boldsymbol{k}\cdot x}, $$
and using the continuity of the operation $x\to g+x$ ,$\forall g\in \langle \boldsymbol{\alpha}_1,\ldots,\boldsymbol{\alpha}_d\rangle$, we have that 
$$f(gx)= \sum_{\boldsymbol{k}\in \Z^N} c_{\boldsymbol{k}} e^{2\pi i \boldsymbol{k}\cdot g}e^{2\pi i \boldsymbol{k}\cdot x}.$$
By the invariance of $f$, we conclude that $\forall \boldsymbol{k} \in \Z^N$, $c_{\boldsymbol{k}}=c_{\boldsymbol{k}} e^{2\pi i \boldsymbol{k}\cdot g}. $ As $f$ is not constant, it exists $\boldsymbol{k}\in \Z^N\setminus\{0\}$ such that $c_{\boldsymbol{k}}\neq 0$, and therefore $\boldsymbol{k}\cdot g\in \Z$, $\forall g\in \langle \boldsymbol{\alpha}_1,\ldots,\boldsymbol{\alpha}_d\rangle$ which contradicts the hypothesis. 
\end{proof}
In what follows, we recall total ergodicity and show that $\Z^d$-torus rotations are totally ergodic.
\begin{defn}
A dynamical system $(Y,T_1,\ldots,T_d)$ is totally ergodic if for all $m\in \N^d$ the system $(Y,T_1^{m_1},\ldots,T_d^{m_d})$ is ergodic. Equivalently, if the action of any finite index subgroup of $\Z^d$ is ergodic. 
\end{defn}

\cref{teo3.1} immediately gives: 

\begin{prop}\label{total-ergodicity-torus-actions}
An ergodic $\Z^d$-rotation $(\T^N,\boldsymbol{\alpha}_1,\ldots,\boldsymbol{\alpha}_d)$ is totally ergodic.
\end{prop}
\begin{proof}
Let $m\in \N^d$. Suppose that $(\T^N,m_1\boldsymbol{\alpha}_1,\ldots,m_d\boldsymbol{\alpha}_d)$ is not ergodic. Then, by \cref{teo3.1}, there is $\boldsymbol{k}\in \Z^N\setminus\{0\}$ such that 
$\forall i\in [d], ~~\boldsymbol{k}\cdot m_i\boldsymbol{\alpha}_i\in \Z$.
In particular, 
$$\forall i\in [d], ~~(\boldsymbol{k}\prod_{j=1}^d m_j)\cdot \boldsymbol{\alpha}_i\in \Z,$$ 
 and by \cref{teo3.1}, $(\T^N,\boldsymbol{\alpha}_1,\ldots,\boldsymbol{\alpha}_d)$ is not ergodic, which is a contradiction.
\end{proof}

We will prove that total ergodicity is equivalent to connectedness in nilsystems. We need the following lemma.

\begin{lema}\label{connected-components-partition}
Let $(X=G/\Gamma,T_1,\ldots,T_d)$ be a $\Z^d$ -nilsystem and $X_0$ be the connected component of $e_X$. Then, there exist $p_1,\ldots,p_d\in\N$ such that $X_0$ is a $(T_1^{p_1},\ldots, T_d^{p_d})$-invariant clopen subset of $X$.
\end{lema}
\begin{proof}
Let $\mu$ be the Haar measure on $X$, which has full support. As $X_0$ is open, we have that for each $i\in [d]$, there exists a minimal period $p_i\in \N$ for which $\tau_i^{p_i}X_0=X_0$, otherwise the sets $(\tau_i^kX_0)_{k\in \N}$ are pairwise disjoint and with the same measure $\mu(X_0)$, which is a contradiction. Given that $\{\tau_i^{j}X_0\}_{j=1}^{p_i}$ partitions $X$,
    and that $X_0$ is open, the set $X_0$ is also closed. Finally, as $\tau_i^{p_i}X_0=X_0$, we have the invariance, and we conclude. 
\end{proof}

\begin{obs}\label{quasi-affine-partition}
Notice that $(X_0,T_1^{p_1},\ldots, T_d^{p_d})$ is a connected $\Z^d$-nilsystem. Furthermore, $X$ is the union of finitely many isomorphic copies of $(X_0,T_1^{p_1},\ldots,T_d^{p_d})$, which are permuted under the dynamics. 
\end{obs}

Finally, we have the following characterization of total ergodicity in a $\Z^d$-nilsystem.

\begin{cor}\label{coro7HKM}
Let $(X,T_1,\ldots,T_d)$ be an ergodic $s$-step $\Z^d$-nilsystem for the Haar measure $\mu$. Then, $X$ is connected if and only if $(X,T_1,\ldots,T_d)$ is totally ergodic for $\mu$.
\end{cor}
\begin{proof}
If $X$ is connected, then by \cref{connected-implies-totally} we have that for $k_1,\ldots,k_d\in \N$, $(X,T_1^{k_1},\ldots,T_d^{k_d})$ is ergodic for $\mu$ if and only if the rotations by $\tau_1^{k_1},\ldots,\tau_d^{k_d}$ are ergodic in the maximal factor torus, which is always true given that such system is totally ergodic for $\mu$ by \cref{total-ergodicity-torus-actions}. On the other hand, if the system is totally ergodic for $\mu$, by \cref{connected-components-partition}, we derive that $X_0$ is a $(T_1^{p_1},\ldots, T_d^{p_d})$-invariant closed subset of $X$. But, as $(X,T_1^{p_1},\ldots, T_d^{p_d})$ is ergodic for $\mu$, thus minimal by \cref{teo11HKM}, we have that $X=X_0$, so $X$ is  connected.
\end{proof}

The following lemma will be useful in subsequent proofs.

\begin{lema}\label{transtivityofnil}
Let $(X,T_1,\ldots,T_d)$ be an $s$-step $\Z^d$-nilsystem and $\boldsymbol{K}\in \R^{d\times d}$ a real matrix with integer entries. Define $S_i=T_1^{K_{i,1}}\circ\cdots\circ T_d^{K_{i,d}},$ for each $i\in [d]$. Then, $(X, S_1,\ldots, S_d)$ is an $s$-step $\Z^d$-nilsystem. In addition, if $\boldsymbol{K}$ is invertible and $(X,T_1,\ldots,T_d)$ is connected and minimal, then $(X, S_1,\ldots, S_d)$ is connected and minimal as well.
\end{lema}
\begin{proof}
Given a matrix $\boldsymbol{K}\in \Z^{d\times d}$ and $(X,T_1,\ldots,T_d)$ an $s$-step $\Z^d$-nilsystem, it is clear that $(X,(T_1^{K_{i,1}}\circ\cdots\circ T_d^{K_{i,d}})_{i=1}^d)$ is still an $s$-step $\Z^d$-nilsystem, as $X$ is an $s$-step nilmanifold and the transformations $(T_1^{K_{i,1}}\circ\cdots\circ T_d^{K_{i,d}})_{i=1}^d$ are given by the corresponding rotations. For the second part of the statement, suppose that $(X,T_1,\ldots,T_d)$ is connected and minimal, and that $\boldsymbol{K}$ is invertible. Let $N\in \N$ be such that $N\cdot K^{-T}\in \Z^{d\times d}$. Then, notice that the action of $S_i$ corresponds to the action generated by the vector $\boldsymbol{s}_i= \boldsymbol{K}^T\boldsymbol{e}_i$ of $\Z^d$, so $N\cdot \boldsymbol{e}_i= N\cdot \boldsymbol{K}^{-T}\boldsymbol{s}_i\in \Z^d$, where $\boldsymbol{e}_i$ denotes the $i$-th canonical vector of $\R^d$. Hence,
$$\langle S_1,\ldots,S_d\rangle=\langle(T_1^{K_{i,1}}\circ\cdots\circ T_d^{K_{i,d}})_{i=1}^d\rangle\supseteq \langle T_1^{N},\cdots,T_d^{N}\rangle,$$
and thus the action of $(S_1,\ldots, S_d)$ is minimal, given that the system $(X,T_1,\ldots,T_d)$ is totally ergodic by \cref{coro7HKM}.
\end{proof}

To finish this subsection, we prove the proposition that allows us to reduce to the connected case in general.

\begin{prop}\label{prop2}
Let $(X,T_1,\ldots, T_d)$ be a minimal $s$-step $\Z^d$-nilsystem. Then, there exists an invertible matrix $\boldsymbol{K}\in \N^{d\times d}$ such that $(X_0,(T_1^{K_{i,1}}\circ\cdots\circ T_d^{K_{i,d}})_{i=1}^d)$ is a connected minimal $s$-step $\Z^d$-nilsystem, where $X_0$ is the connected component of $e_X$. 
\end{prop}
\begin{proof}
By \cref{connected-components-partition} and \cref{transtivityofnil}, there are $p_1,\ldots,p_d\in \N$ such that \hfill\break $(X_0,T_1^{p_1},\ldots, T_d^{p_d})$ is a nilsystem. However, $(X_0,T_1^{p_1},\ldots, T_d^{p_d})$ is not necessarily minimal. To fix this, we will define an upper triangular matrix $\boldsymbol{K}\in \N^{d\times d}$ as follows. Set 
$$K_{i,i}=\min\{ k_i\in [p_i]  :  \exists k_{i+1},\ldots k_{d}\in \Z,  \tau_j^{k_i}\cdots \tau_d^{k_d}X_0=X_0\},$$
and for $j> i$, assuming that we have defined $K_{i,l}$ for every $l\in \{i,\ldots,j-1\}$, we define $K_{i,j}$ by
$$K_{i,j}=\min\{ k_j\in [p_j]  :  \exists k_{j+1},\ldots k_{d}\in \Z,  \tau_i^{K_{i,i}}\cdots \tau_{j-1}^{K_{i,j-1}}\tau_j^{k_j}\cdots \tau_d^{k_d}X_0=X_0\}.  $$

We claim that $(X_0,(T_1^{K_{i,1}}\circ\cdots\circ T_d^{K_{i,d}})_{i=1}^d)$ is a minimal nilsystem. As $(X,T_1,\ldots,T_d)$ is a minimal nilsystem, we only need to prove that
$$\{\tau_1^{l_1}\cdots \tau_d^{l_d}  :  l_1,\ldots,l_d\in \Z,  \tau_1^{l_1}\cdots \tau_d^{l_d} X_0=X_0\}\subseteq  H=\langle (\tau_1^{K_{i,1}}\circ\cdots\circ \tau_d^{K_{i,d}})_{i=1}^d \rangle.  $$

We affirm that for all $i\in [d]$, there are $l_{i,i},l_{i,i+1},\ldots, l_{i,d}\in \Z$ such that
$$\tau_1^{l_1}\cdots \tau_d^{l_d} \equiv \tau_i^{l_{i,i}}\cdots \tau_d^{l_{i,d}} ~\text{ mod } H. $$
Indeed, we assume that this holds for $i<d$, and let $m\in \Z$ and $r\in \{0,\ldots, K_{i,i}-1\}$ such that $l_{i,i}=mK_{i,i}+r$. We define $l_{i+1,j}=l_{i,j}-mK_{i,j}$ for all $ j\geq i+1$ and, as $K_{i,j}=0$ for all $j<i$, we have that
$$\tau_i^{l_{i,i}}\cdots \tau_d^{l_{i,d}} \equiv \tau_{i}^{r}\tau_{i+1}^{l_{i+1,i+1}}\cdots \tau_{i+1}^{l_{i+1,d}} ~\text{ mod } H. $$
Note that $ \tau_{i}^{r}\tau_{i+1}^{l_{i+1,i+1}}\cdots \tau_{i+1}^{l_{i+1,d}}X_0=X_0$, hence, by the minimality of $K_{i,i}$, we obtain that $r=0$, concluding the induction. In this light
$$\tau_{1}^{l_1}\cdots \tau_d^{l_d}\equiv 0  ~\text{ mod } H,$$
which implies that $\tau_{1}^{l_1}\cdots \tau_d^{l_d}\in H$. 
\end{proof}

\begin{obs} The previous proof is based in the idea that there is a parallelepiped that generates all the mesh associated to the dynamics of $\{T_i\}_{i=1}^d$ in $X_0$ (that parallelepiped repeats itself periodically in $\Z^d$). 
\end{obs}
With this and \cref{prop3.4}, it will be possible to reduce to the case in that $(X,T_1,\ldots, T_d)$ is an $s$-step nilsystem with $X$ connected.

\subsection{Nilsystems and the property of complete independence}

This section is devoted to showing how we reduce to the case where the set of Bohr recurrence has the property of complete independence. First, we prove that if we have a redundant set of Bohr recurrence we can reduce the dimensionality of the system.

\begin{prop}\label{eliminating-redundant-coordinates}
Let $(X=G/\Gamma,T_1,\ldots,T_d)$ be a minimal $s$-step $\Z^d$-nilsystem and let $R\subseteq \Z^d$ be a set of $\Z^d$-Bohr recurrence. If $R$ is redundant, then there exist $R'\subseteq \Z^{d-1}$, a set of $\Z^{d-1}$-Bohr recurrence, and a minimal $s$-step $\Z^{d-1}$-nilsystem $(Y=G^Y/\Gamma^Y,T'_1,\ldots,T'_{d-1})$ such that $G^Y$ is a subgroup of $G$ and if $R'$ is a set of recurrence for $(Y,T'_1,\ldots,T'_{d-1})$, then $R$ is a set of recurrence for $(X,T_1,\ldots,T_d)$.
\end{prop}
\begin{proof}
Suppose $R\subseteq \Z^d$ is redundant and let $\boldsymbol{v}=(v_1,\ldots,v_d)\in \Z^d\setminus\{0\}$ given by the definition of redundancy. Replacing $R$ by $R_v=\{ (n_1,\ldots,n_d) \in R  :  \boldsymbol{v}^T\cdot n=0\}$, we may assume without loss of generality that for all $n\in R$, $\boldsymbol{v}^T\cdot n=0$. Furthermore, as $v\neq 0$ we assume without loss that $v_d=1$. Define the integer matrix $M\in \Z^{d\times d}$ by
$$M_{i,j}=\begin{cases} 
            1              & \text{ if } i=j\\
            v_j           & \text{ if } i=d \text{ and } j\neq d \\
           0               & \text{ else } \\
\end{cases}.$$
Clearly, $M$ is invertible. Note that $M^{-1}R\cap \Z^d$ is a set of $\Z^d$-Bohr recurrence by \cref{prop3.4}. Set $R'=\{\boldsymbol{n} \in \Z^{d-1}  : \exists   \boldsymbol{m}\in M^{-1}R\cap \Z^d,~ n_i=m_i \text{ for each } i\in [d-1]\}$,  
$S_i=
T_iT_d^{-v_i}$ for each $i\in [d]$,
and $Y= \overline{\mathcal{O}_{S_{1},\ldots, S_{d-1}}(e_{X})}$. 
Clearly, $(Y,S_1,\ldots,S_{d-1})$ is a minimal $s$-step $\Z^{d-1}$-nilsystem, and by inspecting the proof of \cite[Chapter 11, Theorem 9]{Host_Kra_nilpotent_structures_ergodic_theory:2018}, we know that $Y$ can be represented as $Y=G^Y/\Gamma^Y$ with $G^Y$ a subgroup of $G$. Finally, since, for all $n\in R$ such that $m=M^{-1}n\in\Z^d$ 
$$T_1^{n_1}\cdots T_d^{n_d}= S_1^{m_1}\cdots S_{d}^{m_d},$$
we have that if $R'$ is a set of recurrence for $(Y=G^Y/\Gamma^Y,S_1,\ldots,S_{d-1})$ then $R$ is a set of recurrence for $(X,T_1,\ldots,T_d)$.
\end{proof}
It is immediate that we can apply \cref{eliminating-redundant-coordinates} until we get a non-redundant set of recurrent, as the following corollary shows.
\begin{cor}\label{redundant-coro}
Let $(X=G/\Gamma,T_1,\ldots,T_d)$ be a minimal $s$-step $\Z^d$-nilsystem and $R\subseteq \Z^d$ be a set of $\Z^d$-Bohr recurrence. Then, there exist $d'\leq d$, $R'\subseteq \Z^{d'}$ a non-redundant set of $\Z^{d'}$-Bohr recurrence and $(Y=G^Y/\Gamma^Y,T'_1,\ldots,T'_{d'})$ a minimal $s$-step $\Z^{d'}$-nilsystem such that $G^Y$ is a subgroup of $G$ and if $R'$ is a set of recurrence for $(Y,T'_1,\ldots,T'_{d'})$ then $R$ is of recurrence for $(X,T_1,\ldots,T_d)$. 
\end{cor}
\begin{proof}
By applying \cref{eliminating-redundant-coordinates} iteratively, we get $d'\geq 1$ such that there are $R'\subseteq \Z^{d'}$ a non-redundant set of $\Z^{d'}$-Bohr recurrence and $(Y=G^Y/\Gamma^Y,T'_1,\ldots,T'_{d'})$ a minimal $s$-step $\Z^{d'}$-nilsystem such that $G^Y\leq G$ and if $R'$ is a set of recurrence for $(Y=G^Y/\Gamma^Y,T'_1,\ldots,T'_{d'})$ then $R$ is of recurrence for $(X=G/\Gamma,T_1,\ldots,T_d)$. This is possible by the fact that the process stops at most in $d'=1$. In this condition every set of Bohr recurrence is non-redundant. 
\end{proof}

Now we present the result which allows us to reduce to the case in which we have a set of Bohr recurrence $R\subseteq \Z^d$ with the property of complete independence.

\begin{teo}\label{redefiningR}
Let $R\subseteq \Z^d$ be a non-redundant set of $\Z^d$-Bohr recurrence and $(X,T_1,\ldots,T_d)$ be a connected minimal $s$-step $\Z^d$-nilsystem. Then, there exist a non-redundant, essential, and ordered set of $\Z^{d}$-Bohr recurrence $\tilde{R}\subseteq \Z^{d}$, a connected minimal $s$-step $\Z^{d}$-nilsystem $(X,S_1,\ldots,S_{d})$, and $P\in \mathcal{BC}(\tilde{R})$ satisfying \cref{Complete-independence}, such that if $\tilde{R}$ is a set of recurrence for $(X,S_1,\ldots,S_{d})$, then $R$ is a set of recurrence for $(X,T_1,\ldots,T_d)$.
\end{teo}
\begin{proof}
Using \cref{essential-sufficiency}, and possibly replacing $R$ by a subset of it, we may assume that $R$ is essential. We also assume that $R$ is ordered by \cref{Borh-ordered}. Let $P\in \mathcal{BC}(R)$. We will follow an induction process, changing $(R,P,(X,T_1,\ldots,T_d))$ iteratively. We start with $R^1=R$ and $P^1=P$, $T_{i,1}=T_i$ for all $i\in [d]$. For $N<d$ suppose that we have set $(R^N,P^N,(X,T_{1,N},\ldots,T_{d,N}))$ such that $R^N$ is a non-redundant, essential and ordered set of $\Z^d$-Bohr recurrence, and $P^N\in \mathcal{BC}(R^N)$ satisfies
\begin{equation}\label{CI-Induction}
\forall i\in[N], \{ P_{i,j}^N  :   P_{i,j}^N\neq 0,~ \forall i\leq j\leq N \} \text{  are rationally independent}.   
\end{equation}
In addition, suppose that if $R^N$ is a set of recurrence for $(X,T_{1,N},\ldots,T_{d,N})$, then $R$ is a set of recurrence for $(X,T_1,\ldots,T_d)$. 

We extend this to $N+1$ as follows: If
$$
\forall i\in[N+1], \{ P_{i,j}^N  :   P_{i,j}^N\neq 0,~ \forall i\leq j\leq N+1 \}   
$$
 are rationally independent, we set $(R^{N+1},P^{N+1},(X,T_{1,N+1},\ldots,T_{d,N+1}))=(R^N,P^N,(X,T_{1,N}$ $,\ldots,T_{d,N}))$, extending in this way the hypotheses to $N+1$. If not, we repeat the following process inductively:\\
 
 Define $L_N=\{ j>N \mid P_{N,j}\neq 0$\} and $l_N=\min\{ l\leq N+1  :  P_{l,N+1}^N\neq 0\}$. Notice that $\{ j \in [N+1]  :  P_{l_N,j}^N\neq 0\}=\{l_N,\ldots,N+1\}$ by \cref{consistentency}. Since the coordinates in $R^N$ are ordered, the set $L_N$ corresponds to the coordinates beyond $N$ that correlate with the $N$-th coordinate. Furthermore, if $L_N$ is nonempty, it must contain $N+1$ by \cref{consistentency}. On the other hand, $l_N$ represents the minimal coordinate $j\leq N+1$ that has a nonzero correlation with $N+1$.
 
As $\{P_{l_N,j}^N\}_{j=l_N}^{N+1}$ are not rationally independent, there exists $\boldsymbol{v}\in \Z^{d}\setminus \{\boldsymbol{0}\}$ with $v_j=0$ for all $j<l_N$ and $v_{N+1}\neq 0$, such that
\begin{equation}\label{reconfiguring-the-correlations}
\sum_{j=l_N}^{N+1}v_jP_{l_N,j}^N=0.
\end{equation}       
Define the real matrix $M\in \R^{d\times d}$ with integer entries by
$$M_{i,j}=\begin{cases} 
v_{N+1}                & \text{ if } i=j\\
-v_j       & \text{ if } i=N+1 \text{ and }  j\in \{l_N,\ldots,N+1\} \\
           0               & \text{ else } \\
\end{cases} .$$
Clearly, $M$ is invertible, with inverse
$$M^{-1}_{i,j}=\begin{cases} 
1/v_{N+1}                & \text{ if } i=j\\
v_j/v_{N+1}^2        & \text{ if } i=N+1 \text{ and }  j\in \{l_N,\ldots,N\}\\
           0               & \text{ else } \\
\end{cases} .$$

We set $\epsilon=(\min_{j\in \{l_N,\ldots,N+1\}} |P_{l_N,j}^N|)/(2d ||\boldsymbol{v}||_\infty)\in (0,1)$ where $||\boldsymbol{v}||_\infty=\max_{j\in [d]} |v_j|$.  Notice that $R'^{N}:=M^{-1}R^N_{P^N,\epsilon}\cap \Z^d$ is a non-redundant set of $\Z^d$-Bohr recurrence by \cref{prop3.4} and \cref{prop3.4-v2}, where
\begin{equation}\label{defRPeps}
  R_{P^N,\epsilon}^N=\left\{ \boldsymbol{n}\in R \mid |\frac{n_j}{n_i}-P_{i,j}^N|<\epsilon,~ \forall j\geq i\right\}.  
\end{equation}
We may further assume that $R'^N$ it is essential without loss of generality.

We also define 
$$T_{j,N}'=\begin{cases}   T_j^{v_{N+1}}T_{N+1}^{-v_j} & \text{if }j\in\{l_N,\ldots,N\}\\
T_j^{v_{N+1}} &\text{ else}\end{cases}. $$
By \cref{transtivityofnil}, $(X,T_{1,N}',\ldots,T_{d,N}')$ is a connected minimal $s$-step nilsystem. For $n\in R^N$ such that $m=M^{-1}n \in \Z^d$, we have
\begin{equation}\label{recurrence-bond}
T_{1,N}^{n_1}\cdots T_{d,N}^{n_d}= T_{1,N}'^{m_1}\cdots T_{d,N}'^{m_d}.
\end{equation}

Additionally, denote $I_N=\{(i,j)\in [d]^2 : j\geq i\text{ and } i\neq N+1\}$ and define $P'^N\in[-1,1]^{I}$ such that for each $(i,j)\in I$
$$P'^N_{i,j}= \begin{cases}
    P^N_{i,j} & \text{ if } j\neq N+1\\
    0 & \text{ if } j=N+1
\end{cases}.  $$
We observe that $R'^{N}$ can be still assumed to be ordered, giving that for every $\boldsymbol{m}\in R'^N$ there is $n\in R^N_{P,\epsilon}$ such that $$\boldsymbol{m}=\Bigl(\frac{n_1}{v_{N+1}},\ldots,\frac{n_N}{v_{N+1}},\sum_{j=l_N}^{N+1} \frac{v_j}{v_{N+1}^2} n_j,\frac{n_{N+2}}{v_{N+1}},\cdots, \frac{n_{d}}{v_{N+1}}\Bigr), $$
and thus the only coordinate possibly changing order is the $(N+1)$-th coordinate, which can only move to the right (i.e. it can only get relatively smaller in the new order of $R'^N$), given that
\begin{align*}
   \left| \frac{\sum_{j=l_N}^{N+1} \frac{v_j}{v_{N+1}^2} n_j}{\frac{n_N}{v_{N+1}}} \right|&= \left| \sum_{j=l_N}^{N+1} \frac{v_j}{v_{N+1}} \frac{n_j}{n_N}\right|\\
   &= \left|\frac{n_{l_N}}{n_N }\right| \cdot \frac{1}{|v_{N+1}|}\cdot \left|\sum_{j=l_N}^{N+1}v_j \frac{n_j}{n_{l_N}}\right|\\
   &\leq\frac{2}{|P_{l_N,N}|}\cdot  \sum_{j=l_N}^{N+1} \norm{\boldsymbol{v}}_\infty |\frac{n_j}{n_{l_N}}- P_{l_N,j} | <1
\end{align*}
where we used \cref{reconfiguring-the-correlations}, \cref{defRPeps}, and that 
$$|\frac{n_{N}}{n_{l_N}}-P_{l_N,N} |< \epsilon \Longrightarrow |P_{l_N,N}|- \frac{|n_{N}|}{|n_{l_N}|}< \frac{|P_{l_N,N}|}{2} \Longrightarrow  \frac{|n_{l_N}|}{|n_{N}|}<\frac{2}{|P_{l_N,N}|}. $$
This being so, for each $\boldsymbol{m}\in R'^N$, $|m_1|\leq \cdots \leq |m_N|<|m_{N+1}|$ and $|m_N|\leq |m_{N+2}|\leq \cdots |m_d|$, so we can assume $R'^N$ ordered by rearranging the $(N+1)$-th coordinate of $R'^{N}$, $P'^N$, and $(X,T_{1,N}',\ldots,T_{d,N}')$ (possibly reducing to a subset of $R'^{N}$). In addition, rearranging the coordinates of $I$ as well, we have that $P'^N$ is a vector of $I$-Bohr correlations and using \cref{Bohrcompleteness} we can extend $P'^N$ we can assume that $P'^N$ is a vector of Bohr correlations. Finally, if $R'^{N}$ is a set of recurrence for $(X,T_{1,N}',\ldots,T_{d,N}'),$ then $R^N$ is a set of recurrence for $(X,T_{1,N},\ldots,T_{d,N})$ by \eqref{recurrence-bond}. Hence, we can replace $(R^N,P^N,(X,T_{1,N},\ldots,T_{d,N}))$ by $(R'^N,P'^N,(X,T_{1,N}',\ldots,T_{d,N}'))$. We also define $L_N'=\{ j>N \mid P_{N,j}'^N\neq 0$\} and $l_N'=\min\{ l\leq N+1   :  P_{l,N+1}'^N\neq 0\}$. We note that $|L_N'|<|L_N|$ as $L_N'\subseteq L_N$ and $P'^{N+1}_{j,c_{N+1}}=0$ for all $j\in [N]$ by \cref{reconfiguring-the-correlations}, where $c_{N+1}\in [d]$ is the coordinate designated to the previous $(N+1)$-th coordinate. Subsequently, if $\{P_{l_N',j}'^{N}\}_{j=l_N'}^{N+1}$ are not rationally independent, the process is repeated.

Since each iteration of the process reduces the cardinality of $|L_N|$ by one, the process concludes after finitely many iterations, yielding the desired result.
\end{proof}

\subsection{The Main Theorem}\label{5.4}
From now on, we assume $s\geq 2$. For an $s$-step $\Z^d$-nilsystem $(X,T_1,\ldots,T_d)$ denote by
\begin{equation}\label{1}
\tilde{G}=G/G_s,~~\tilde{\Gamma}=\Gamma/(\Gamma\cap G_s)~~ \text{ and } ~ \tilde{X}=\tilde{G}/\tilde{\Gamma}.
\end{equation}
Then, $\tilde{G}$ is an $(s-1)$-step nilpotent Lie group, $\tilde{\Gamma}$ is a discrete cocompact subgroup, $\tilde{X}$ is an $(s-1)$-nilmanifold and the quotient map $G\to \tilde{G}$ induces a projection $\pi\colon X\to \tilde{X}$. Therefore, we can view $\tilde{X}$ as the quotient of $X$ under the action of $G_s$. Let $\tilde{\tau_1},\ldots,\tilde{\tau_d}$ be the image of $\tau_1,\ldots,\tau_d$ in $\tilde{G}$ and $\tilde{T}_1,\ldots,\tilde{T}_d$ be the rotations by $\tilde{\tau_1},\ldots,\tilde{\tau_d}$ in $\tilde{X}$. Then, $(\tilde{X},\tilde{T_1},\ldots,\tilde{T_d})$ is an $(s-1)$-nilsystem and $\pi\colon X\to\tilde{X}$ is a factor map.\\

Finally, we are able to prove \cref{Theorem-A}.
\begin{customthm}{A}
Let $R\subseteq \Z^d$ be a set of $\Z^d$-Bohr recurrence. Then, for every integer $s\geq 1$, $R$ is a set of recurrence for every minimal $s$-step $\Z^d$-Weyl system.
\end{customthm}

\begin{proof}
We proceed by induction on $s$. If $s=1$, the result is direct. Henceforth, we assume that $s\geq 2$, and that the statement holds for $(s-1)$-step $\Z^d$-Weyl systems. Let $R\subseteq \Z^d$ be a set of $\Z^d$-Bohr recurrence and let $(X=G/\Gamma,T_1,\ldots, T_d)$ be a minimal $s$-step $\Z^d$-Weyl system. First, we can assume that $R$ is non-redundant, by \cref{redundant-coro}, and replacing the pair $(R,(X=G/\Gamma,T_1,\ldots, T_d))$ by another pair \break  $(R',(Y=G^Y/\Gamma^Y,T'_1,\ldots,T'_{d'}))$, where $d'\leq d$, $R'\subseteq \Z^{d'}$ is a set of $\Z^{d'}$-Bohr recurrence and $(Y,T'_1,\ldots,T'_{d'})$ is a minimal $s$-step $\Z^{d'}$-nilsystem. Furthermore, $(Y,T'_1,\ldots,T'_{d'})$ is a Weyl system, as $G^Y$ being a subgroup of $G$ implies that $G_0^Y$ is abelian, just like $G_0$.

Next, we assume without loss of generality that $X$ is connected and therefore, after a topological conjugacy using \cref{affine-equals-G0-abelian}, affine. Indeed, by \cref{prop2} we have that there exists an invertible matrix $\boldsymbol{K}\subseteq \N^{d\times d}$ such that $(X_0,(T_1^{K_{i,1}}\circ \cdots \circ T_d^{K_{i,d}})_{i=1}^d)$ is a minimal affine nilsystem. On the other hand, by \cref{prop3.4} the set $R_0=\{\boldsymbol{n}\in \Z^d  :  \boldsymbol{K}\cdot\boldsymbol{n} \in R\} $ is a non-redundant set of $\Z^d$-Bohr recurrence. Substituting $X$ by $X_0$ and $R$ by $R_0$, we reduce to the case where $X$ is connected. This is possible due to $R_0$ being a set of recurrence for $X_0$ implies that $R$ is a set of recurrence for $X$. In fact, as we mentioned in \cref{quasi-affine-partition}, $X$ is partitioned by finite translations of $X_0$. So, for an open subset $V\subseteq X$, we have that there are $p_1,\cdots,p_d\in \N$ such that $V\cap T_1^{p_1}\cdots T_d^{p_d}X_0\neq \emptyset$. In particular, the open set $V'=T_1^{-p_1}\cdots T_d^{-p_d}V\cap X_0$ is nonempty. If $R_0$ is a set of recurrence for $X_0$, then there is $\boldsymbol{n}\in R_0$ such that $V'\cap \prod_{i=1}^d(T_1^{K_{i,1}}\circ \cdots \circ T_d^{K_{i,d}})^{n_i}V'\neq \emptyset$, which implies  $V\cap T_1^{-(\boldsymbol{K}\cdot\boldsymbol{n})_1}\cdots T_d^{-(\boldsymbol{K}\cdot\boldsymbol{n})_d}V\neq \emptyset$. As $\boldsymbol{K}\cdot\boldsymbol{n} \in R$, we have that $R$ would be a set of recurrence for $X$, concluding that is enough to assume that $(X,T_1,\ldots,T_d)$ is affine.

Thanks to \cref{redefiningR}, we can also assume that $R$ is essential and ordered with $P\in \mathcal{BC}(R)$ satisfying \cref{Complete-independence}, changing the system $(X=G/\Gamma,T_1,\ldots, T_d)$ to another connected minimal $s$-step $\Z^d$-affine nilsystem.\\

Let $U$ be a nonempty open subset of $X$. As $X$ is minimal, by translating $U$ to the origin using the transformations $(T_1,\ldots,T_d)$, we can assume without loss of generality that $U$ is the open ball $B(e_X,3\epsilon)$ for some $\epsilon>0$. Let $R_\epsilon\subseteq R$ given by \cref{lema-previo-resultado-principal}. Let $\pi\colon X\to \tilde{X}=X/G_s$ be the canonical factor map and $(\tilde{X},\tilde{T}_1,\ldots ,\tilde{T}_d)$ the $(s-1)$-nilsystem discussed at the beginning of \cref{5.4}. Since $(\tilde{X},\tilde{T}_1,\ldots ,\tilde{T}_d)$ is an $(s-1)$-step affine nilsystem, it follows from the induction hypothesis, \cref{Bands-Prop}, and the fact that $\pi$ is open by \cref{open-nilmap}, that there exist arbitrarily large elements $\boldsymbol{n}\in R_{\epsilon}$ with $\pi(B(e_X,\epsilon)) \cap \tilde{T}_1^{-n_1}\cdots \tilde{T}_d^{-n_d}\pi(B(e_X,\epsilon)) \neq \emptyset$. It follows that there exist $x\in X$ and $v\in G_s$ with $d_X(x,e_X)<\epsilon$ and $d_X(T_1^{n_1}\cdots T_d^{n_d}x,v\cdot e_X)<\epsilon$. Lifting $x$ to $G$, we obtain $g\in G$ and $\gamma\in \Gamma$ with
\begin{equation}\label{inductive-hypothesis}
  d_G(g,e_G)<\epsilon \text{  and  } d_G(\tau_1^{n_1}\cdots \tau_d^{n_d} g, v\gamma)<\epsilon.   
\end{equation}

Using \cref{lema-previo-resultado-principal} to approximate $v^{-1}\in G_s$ and the fact that $\boldsymbol{n}\in R_{\epsilon}$ yield the existence of $\boldsymbol{h}\in G_{s-1}$ such that
\begin{equation}\label{First-eps}
  d_G(\boldsymbol{h},e_G)< \epsilon \text{  and  } d_X([\boldsymbol{h}^{-1},\tau_1^{n_1}]\cdots [\boldsymbol{h}^{-1},\tau_d^{n_d}]\Gamma,v^{-1}\Gamma)<\epsilon.
\end{equation}
In particular, lifting the last inequality to $G$ we obtain $\theta\in \Gamma\cap G_s$ such that 
\begin{equation}\label{claim-eq}
d_G([\boldsymbol{h}^{-1},\tau_1^{n_1}]\cdots [\boldsymbol{h}^{-1},\tau_d^{n_d}]
    ,v^{-1}\theta)<\epsilon.
\end{equation}
Write $y=\boldsymbol{h} \cdot x$, and note that $y$ is the projection of $\boldsymbol{h}g$ in $X
$ and 
$$d_X(y,e_X)  \leq d_G(\boldsymbol{h}g,e_G)\leq d_G(\boldsymbol{h},e_G) + d_G(g,e_G)<2\epsilon. $$

Using that $[\boldsymbol{h}^{-1},\tau_1^{n_1}],\ldots, [\boldsymbol{h}^{-1},\tau_d^{n_d}]\in G_s$ we get
\begin{equation*}
          d_X(T_1^{n_1}\cdots T_d^{n_d}y,e_X)\leq d_G(\tau_1^{n_1}\cdots \tau_d^{n_d}\boldsymbol{h}g,\theta\gamma)= d_G(\boldsymbol{h}[\boldsymbol{h}^{-1},\tau_1^{n_1}]\cdots [\boldsymbol{h}^{-1},\tau_d^{n_d}]\tau_1^{n_1}\cdots \tau_d^{n_d}g,\theta\gamma).
\end{equation*}
Using the triangle inequality, the right invariance of the distance $d_G$, and \cref{First-eps} yields 
\begin{align*}
   d_X(T_1^{n_1}\cdots T_d^{n_d}y,e_X)&\leq d_G(\boldsymbol{h},e_G)+d_G([\boldsymbol{h}^{-1},\tau_1^{n_1}]\cdots [\boldsymbol{h}^{-1},\tau_d^{n_d}]\tau_1^{n_1}\cdots \tau_d^{n_d}g,\theta\gamma)  \\
   &\leq \epsilon + d_G( \tau_1^{n_1}\cdots \tau_d^{n_d}g[\boldsymbol{h}^{-1},\tau_1^{n_1}]\cdots [\boldsymbol{h}^{-1},\tau_d^{n_d}],\theta\gamma). 
\end{align*}
Invoking again the triangle inequality, \cref{inductive-hypothesis}, and \cref{claim-eq} we get
\begin{align*}
      d_X(T_1^{n_1}\cdots T_d^{n_d}y,e_X)& \leq \epsilon + d_G(\tau_1^{n_1}\cdots \tau_d^{n_d}g[\boldsymbol{h}^{-1},\tau_1^{n_1}]\cdots [\boldsymbol{h}^{-1},\tau_d^{n_d}]
    ,\tau_1^{n_1}\cdots \tau_d^{n_d}gv^{-1}\theta)\\
    &\hspace{0.6cm} +d_G(\tau_1^{n_1}\cdots \tau_d^{n_d}gv^{-1}\theta,\theta\gamma)  \\
     &=\epsilon + d_G([\boldsymbol{h}^{-1},\tau_1^{n_1}]\cdots [\boldsymbol{h}^{-1},\tau_d^{n_d}]  ,v^{-1}\theta)+d_G(\tau_1^{n_1}\cdots \tau_d^{n_d}g,v\gamma)<3\epsilon ,
\end{align*}
finishing the proof. 
\end{proof}

\end{document}